\documentclass[11pt, oneside]{amsart}
\usepackage{amsfonts, amstext, amsmath, amsthm, amscd, amssymb}
\usepackage{manfnt}

\usepackage{url}
\usepackage[all]{xypic}

\usepackage{pstricks}
\usepackage{pstricks-add}
\usepackage{multido}

\usepackage{marginnote}

\usepackage[paper=a4paper, text={138mm,208mm},centering]{geometry}

\usepackage{graphics, pinlabel, color}
\usepackage{comment}

\usepackage[all,graph]{xy}
\usepackage[bookmarks]{hyperref}
\usepackage{float, placeins, booktabs, longtable}

\renewcommand{\setminus}{{\smallsetminus}}

\newcommand{\bp}{\begin{pmatrix}}
\newcommand{\ep}{\end{pmatrix}}
\newcommand{\be}{\begin{equation}}
\newcommand{\ee}{\end{equation}}
\newcommand{\ol}[1]{\overline{#1}}
\newcommand{\e}{\varepsilon}

\numberwithin{equation}{section}

\theoremstyle{plain}
\newtheorem{theorem}[equation]{Theorem}
\newtheorem{lemma}[equation]{Lemma}
\newtheorem{proposition}[equation]{Proposition}

\theoremstyle{definition}

\newtheorem{remark}[equation]{Remark}

\newtheorem{definition}[equation]{Definition}

\numberwithin{equation}{section}

 \newtheoremstyle{TheoremNum}
        {}{}              
        {\itshape}                      
        {}                              
        {\bfseries}                     
        {.}                             
        { }                             
        {\thmname{#1}\thmnote{ \bfseries #3}}
\theoremstyle{TheoremNum}

\newtheorem{thmn}{Theorem}

\def\Z{\mathbb Z}
\def\R{\mathbb R}

\def\O{\Omega}
\def\wt#1{\widetilde{#1}}

\def\wtr{\rho}
\def\etp{\vartheta}
\def\taup{\tau}
\def\taus{\varsigma}

\def\p{\partial}

\def\vu{\vec{u}}
\def\vw{\vec{w}}
\def\sm{\setminus}
\def\S{\Sigma}

\DeclareMathOperator\Int{Int}
\DeclareMathOperator\codim{codim}

\DeclareMathOperator\cl{cl}
\newrgbcolor{lightgreen}{0.95 1.0 0.95}
\newrgbcolor{lightgreen}{0.95 1.0 0.95}
\newrgbcolor{lightgreenbutnotsolight}{0.85 1.0 0.85}
\newrgbcolor{darkgreen}{0.2 0.6 0.2}
\newrgbcolor{lightblue}{0.8 0.8 1.0}
\newrgbcolor{moderateblue}{0.5 0.5 1.0}
\newrgbcolor{orange}{1.0 0.7 0.0}

\iftrue
\makeatletter
\def\@maketitle{%
  \normalfont\normalsize
  \@adminfootnotes
  \@mkboth{\@nx\shortauthors}{\@nx\shorttitle}%
  \global\topskip42\p@\relax 
  \@settitle
  \ifx\@empty\authors \else \@setauthors \fi
  \ifx\@empty\@dedicatory
  \else
    \vskip.5em
    \baselineskip18\p@
    \vtop{\raggedright{\footnotesize\itshape\@dedicatory\@@par}%
      \global\dimen@i\prevdepth}\prevdepth\dimen@i
  \fi
  \@setabstract
  \normalsize
  \if@titlepage
    \newpage
  \else
    \dimen@34\p@ \advance\dimen@-\baselineskip
    \vskip\dimen@\relax
  \fi
}
\def\@settitle{%
  \vspace*{-20pt}
  \begin{flushleft}%
    \baselineskip14\p@\relax
    \normalfont\bfseries\Large
    \@title
  \end{flushleft}%
}
\def\@setauthors{%
  \begingroup
  \def\thanks{\protect\thanks@warning}%
  \trivlist
  \large \@topsep30\p@\relax
  \advance\@topsep by -\baselineskip
  \item\relax
  \author@andify\authors
  \def\\{\protect\linebreak}%
  \authors
  \ifx\@empty\contribs
  \else
    ,\penalty-3 \space \@setcontribs
    \@closetoccontribs
  \fi
  \normalfont\raggedright
  \@setaddresses
  \endtrivlist
  \endgroup
}
\def\@setaddresses{\par
  \nobreak \begingroup
  \small
  \def\author##1{\nobreak\addvspace\smallskipamount}%
  \def\\{\unskip, \ignorespaces}%
  \interlinepenalty\@M
  \def\address##1##2{\begingroup
    \par\addvspace\bigskipamount\noindent
    \@ifnotempty{##1}{(\ignorespaces##1\unskip) }%
    {\ignorespaces##2}\par\endgroup}%
  \def\curraddr##1##2{\begingroup
    \@ifnotempty{##2}{\nobreak\noindent\curraddrname
      \@ifnotempty{##1}{, \ignorespaces##1\unskip}\/:\space
      ##2\par}\endgroup}%
  \def\email##1##2{\begingroup
    \@ifnotempty{##2}{\nobreak\noindent E-mail address%
      \@ifnotempty{##1}{, \ignorespaces##1\unskip}\/:\space
      \ttfamily##2\par}\endgroup}%
  \def\urladdr##1##2{\begingroup
    \def~{\char`\~}%
    \@ifnotempty{##2}{\nobreak\noindent\urladdrname
      \@ifnotempty{##1}{, \ignorespaces##1\unskip}\/:\space
      \ttfamily##2\par}\endgroup}%
  \addresses
  \endgroup
  \global\let\addresses=\@empty
}
\def\@setabstracta{%
    \ifvoid\abstractbox
  \else
    \skip@25\p@ \advance\skip@-\lastskip
    \advance\skip@-\baselineskip \vskip\skip@
    \box\abstractbox
    \prevdepth\z@ 
    \vskip-10pt
  \fi
}
\renewenvironment{abstract}{%
  \ifx\maketitle\relax
    \ClassWarning{\@classname}{Abstract should precede
      \protect\maketitle\space in AMS document classes; reported}%
  \fi
  \global\setbox\abstractbox=\vtop \bgroup
    \normalfont\small
    \list{}{\labelwidth\z@
      \leftmargin0pc \rightmargin\leftmargin
      \listparindent\normalparindent \itemindent\z@
      \parsep\z@ \@plus\p@
      
    }%
    \item[\hskip\labelsep\bfseries\abstractname.]%
}{%
  \endlist\egroup
  \ifx\@setabstract\relax \@setabstracta \fi
}
\def\section{\@startsection{section}{1}%
  \z@{-1.2\linespacing\@plus-.5\linespacing}{.8\linespacing}%
  {\normalfont\bfseries\large}}
\def\subsection{\@startsection{subsection}{2}%
  \z@{-.8\linespacing\@plus-.3\linespacing}{.3\linespacing\@plus.2\linespacing}%
  {\normalfont\bfseries}}
\def\subsubsection{\@startsection{subsection}{3}%
  \z@{.7\linespacing\@plus.2\linespacing}{-1.5ex}%
  {\normalfont\itshape}}
\def\@secnumfont{\bfseries}
\makeatother
\fi 

\def\to{\mathchoice{\longrightarrow}{\rightarrow}{\rightarrow}{\rightarrow}}
\makeatletter
\newcommand{\shortxra}[2][]{\ext@arrow 0359\rightarrowfill@{#1}{#2}}
\def\longrightarrowfill@{\arrowfill@\relbar\relbar\longrightarrow}
\newcommand{\longxra}[2][]{\ext@arrow 0359\longrightarrowfill@{#1}{#2}}

\makeatother

\makeatletter
\def\Nopagebreak{\@nobreaktrue\nopagebreak}
\makeatother

\begin{document}

\title[Embedded Morse theory]{Embedded Morse theory and relative splitting of cobordisms of manifolds}
\author{Maciej Borodzik}
\address{Institute of Mathematics, University of Warsaw, ul. Banacha 2,
02-097 Warsaw, Poland}
\email{mcboro@mimuw.edu.pl}

\author{Mark Powell}
\address{Department of Mathematics, Indiana University, Bloomington, IN, 47405, USA}
\email{macp@indiana.edu}


\def\subjclassname{\textup{2010} Mathematics Subject Classification}
\expandafter\let\csname subjclassname@1991\endcsname=\subjclassname
\expandafter\let\csname subjclassname@2000\endcsname=\subjclassname

\begin{abstract}
We prove that an embedded cobordism between manifolds with boundary can be split into a sequence of right product and
left product cobordisms, if the codimension of the embedding is at least two.  This is a topological counterpart of the algebraic splitting theorem for embedded cobordisms of the first author, A.~N\'emethi and A.~Ranicki. In the codimension one case, we provide a slightly weaker statement.

We also give proofs of rearrangement and cancellation theorems for handles of embedded submanifolds with boundary.
\end{abstract}
\maketitle

\section{Introduction}

We investigate the Morse theory of embedded cobordisms of manifolds with boundary.  An embedded cobordism $(Z,\O)$ is a cobordism $\O$ between two manifolds $\S_0$ and $\S_1$ in $Z \times [0,1]$, where $\S_i = \O \cap (Z \times \{i\})$ for $i=0,1$.  Both $\S_0$ and $\S_1$ can have nonempty boundary, and $\partial \O \sm (\S_0 \cup \S_1)$ can also be nonempty.  By a small perturbation it can be arranged that the projection map $F \colon Z \times [0,1] \to [0,1]$ restricts to a Morse function $f$ on $\O$.

An instance of an embedded cobordism is when there are embeddings of two closed $(n-2)$-dimensional manifolds $N_0$ and $N_1$ into $Z=S^{n}$.
That is,  $N_0$ and $N_1$ are two non-spherical
links in $S^{n}$, there is a cobordism between the two links and $\S_0,\S_1$ are Seifert surfaces for $N_0$ and $N_1$ respectively. A Pontrjagin--Thom
construction guarantees the existence of $\O$ with the properties above.  Morse theory can be used to study the relation between
Seifert forms associated with $\S_0$ and $\S_1$; see \cite{BNR2}.

Morse theory for manifolds with boundary was studied in the 1970s, by~\cite{Braess} and~\cite{Hajduk} independently.  Recently, it was rediscovered by \cite{KM}
in the context of Floer theory, and since then many articles about Morse theory for manifolds with boundary have appeared; see for instance
\cite{Lau, Blo, BNR1}. A recent application of it is the rapidly developing theory of bordered Heegaard Floer homology; compare \cite{Lipshitz}.

Morse functions on manifolds with boundary have three types of critical point: interior, boundary stable and boundary unstable.  A boundary critical point is stable or unstable according to whether the ascending or descending submanifold of the critical point lies in the boundary (see Definition~\ref{def:stable-unstabl} for more details).

The main result of this paper is the following theorem which shows that the critical points of an embedded cobordism of manifolds with boundary can be split i.e.\ pushed to the boundary, where they become two boundary critical points, when the codimension is at least two.  Thus the cobordism can be expressed as a cobordism with only boundary critical points. This theorem is the topological counterpart of \cite[Main Theorem 1]{BNR2}.
As in \cite{BNR1}, boundary stable critical points correspond to the addition of left half-handles and boundary unstable critical points correspond to the addition of right half-handles.

\begin{thmn}[\ref{thm:main1}\textmd{\,(Global Handle Splitting Theorem)}]
Let $(Z,\O)$ be an embedded cobordism such that $\O\subset Z\times[0,1]$ has codimension 2 or more. Suppose that $\O$, $\S_0$ and $\S_1$ have
no closed connected components. Then there exists a map $F\colon Z\times[0,1]\to [0,1]$, which is homotopic through submersions
(see Definition~\ref{def:regularhomotopy})
to the projection onto the second factor, such that $\O$ can be expressed as a union:
\[\O=\O_{-1/2}\cup \O_0\cup \O_{1/2}\cup \O_1\cup \O_{3/2}\cup\dots\cup \O_{n+1/2}\cup \O_{n+1},\]
where $\O_i=\O\cap F^{-1}([(2i+1)/(2n+4),(2i+2)/(2n+4)])$ and
\begin{itemize}
\item[$\bullet$] $\O_{-1/2}$ is a cobordism given by a sequence of index $0$ handle attachments;
\item[$\bullet$] if $i\in\{0,\dots,n\}$, then $\O_i$ is a right product cobordism given by a sequence of elementary index $i$ right product cobordisms;
\item[$\bullet$] if $i+1/2\in \{1,\dots,n+1\}$, then $\O_i$ is a left product cobordism given by a sequence of elementary index $i+1/2$ left product cobordisms;
\item[$\bullet$] $\O_{n+1}$ is a cobordism given by a sequence of index $n+1$ handle attachments.
\end{itemize}
\end{thmn}
We refer to Definition~\ref{def:elementaryindex} for a simple explanation of elementary index $i$ right/left product cobordisms. A more detailed
description is given in \cite[Section 2]{BNR1}, or \cite{Hajduk,Blo}.

The case of codimension one is stated in Theorem~\ref{prop:codim-one-splitting}.  Due to problems with rearrangement of handles in codimension one, the results in codimension one are slightly weaker, for example we were unable to guarantee that index~$1$ and~$n$ handles split (see Proposition~\ref{prop:codim1-4.11}). Nevertheless the flavour of the result is similar to the result in higher codimensions stated above.

\begin{remark}
Note that Theorem~\ref{thm:main1} is not just a corollary of the rearrangement theorem (Theorem~\ref{thm:egrt} below) and the embedded analogue of the Thom-Milnor theorem that $\O$ can be expressed as a union $\O_1'\cup\ldots\O_k'$,
where $\O_k'$ is a cobordism corresponding to a single critical point (interior or boundary).

The proof of Theorem~\ref{thm:main1}
is more difficult, due to the requirement that all the interior critical points of non-extremal indices are absent.  Being able to push interior critical points of non-extremal indices to the boundary is important when one
wants to compare the Seifert forms as in \cite{BNR2}. Further explanation is given after the statement of Theorem~\ref{prop:codim-one-splitting}.
\end{remark}

We also give results detailing when critical points of Morse functions on embedded manifolds with boundary can be rearranged and cancelled.  These folklore results have been stated in the literature in the case of embedded manifolds with empty boundary, but we could not find detailed proofs.

The statements and proofs are new in the case of embedded manifolds with boundary.  The rough idea is to adapt arguments of Milnor for the absolute case of closed manifolds, from \cite{Mi1, Mi2}, and to understand \emph{ascending and descending membranes} (Definition~\ref{def:membrane}), which appeared in the papers of B.~Perron~\cite{Pe} and R.~Sharpe~\cite{Sh}.  Intersections involving these membranes can obstruct the rearrangement and cancellation of critical points of Morse functions on embedded manifolds, in cases when the operation could be performed in the absolute setting.

In order to state our rearrangement theorem we use the following definition.

\begin{definition}
  A \emph{configuration} $\Xi$ of the critical points of an embedded cobordism $(Z,\O)$ is an assignment of a value $\Xi(z_i) \in (0,1)$ to each critical point $z_i$ of~$f$.

  An \emph{admissible configuration} is a configuration satisfying the following conditions:
\begin{itemize}
\item[(A1)] if $z,w$ are critical points with indices $k$, $l$ with $k<l$, then $\Xi(z)<\Xi(w)$;
\item[(A2)] if $z,w$ have the same index $k$ and if $z$ is boundary stable and $w$ is boundary unstable, then $\Xi(z)<\Xi(w)$.
\end{itemize}
\end{definition}

Next we state our rearrangement theorem, which in the case that the codimension is at least two, essentially says that any admissible configuration can be realised.

\begin{thmn}[\,\ref{thm:egrt}\textmd{\,(Global Rearrangement Theorem)}]
Suppose $(Z,\O)$ is a cobordism.  Given an admissible configuration $\Xi$ of the critical points of $f$,
if $\codim(\O\subset Z\times[0,1])\geqslant 2$, there exists a function $G\colon Z\times[0,1]\to[0,1]$ without critical points, homotopic through submersions to $F$ (Definition~\ref{def:regularhomotopy}), which restricts
to a Morse function $g\colon\O\to[0,1]$, such that $g$ restricted to the critical points agrees with~$\Xi$ (the type and index of each critical point is preserved).

\end{thmn}

Results on rearrangement in codimension one are given in Theorem~\ref{prop:egrt1}.

The cancellation theorem is as follows; here there are no codimension restrictions.

\begin{thmn}[\,\ref{lem:eec}\textmd{\,(Elementary Cancellation Theorem)}]
Let $(Z,\O)$ be a cobordism. Let $z$ and $w$ be critical points of $f$ of indices $k$ and $k+1$, of the same type (i.e. either both interior,
or both boundary stable, or both boundary unstable).
Suppose that $\xi$ is an embedded gradient-like vector field, which is Morse--Smale (Definition~\ref{def:emsc}), and that there exists
a \emph{single trajectory} $\gamma$ of $\xi$ connecting $z$ with $w$. If $z$, $w$ are interior critical points, we require that $\gamma\subset\O$.  If $z$, $w$ are boundary critical points, then we require that $\gamma\subset Y$. Furthermore, suppose that there are no broken trajectories between $z$ and $w$.

Then, for any neighbourhood $U$ of $\gamma$, there
exists a vector field $\xi'$ on $Z\times[0,1]$, agreeing with $\xi$ away from $U$, non-vanishing on $U$, and a function
$F'\colon Z\times[0,1]\to[0,1]$ such that $f'=F'|_\O$ has the same critical points as $f$ with the exception of $z$ and $w$,
which are regular points of $f'$, and such that $\xi'$ is an embedded gradient-like vector field with respect to $F'$ (Definition~\ref{def:eglv}).
\end{thmn}

\subsection*{Organisation of the paper}

Section~\ref{section:review-morse-theory-boundary} reviews Morse theory for manifolds with boundary and gives the definitions and terminology that we will use throughout the paper.  The Embedded Morse Lemma~\ref{lem:eml} is also proven.

 Section~\ref{section:embedded-gl-vector-fields} introduces the notion of embedded gradient-like vector fields, defines membranes, and develops embedded Morse theory.  Section~\ref{section:recovering-morse-function} shows how to recover a Morse function from an embedded gradient-like vector field.  The Embedded Isotopy Lemma~\ref{lem:eil} is proved in Section~\ref{section:isotopy-lemma}.

Section~\ref{section:rearrangement} deals with rearrangement of critical points.  Section~\ref{section:elem-rearrangement-thm} proves the Elementary Rearrangement Theorem~\ref{thm:ert}, Section~\ref{section:morse-smale-condition} looks at the dimensions of transverse intersections of membranes, and Section~\ref{section:global-rearrangement} gives the Global Rearrangement Theorem~\ref{thm:egrt}.

Section~\ref{section:cancellation} looks at cancellation of critical points; the Elementary Cancellation Theorem~\ref{lem:eec} is proven in Section~\ref{section:elem-cancel-thm}.

Section~\ref{section:splitting} proves our main result, concerning the pushing of interior critical points to the boundary, splitting a handle into two half-handles.  The Elementary Handle Splitting Theorem~\ref{thm:embedded} is proven in Section~\ref{section:elem-splitting} and the Global Handle Splitting Theorem~\ref{thm:main1} is proven in Section~\ref{section:global-handle-splitting}.

\subsection*{Acknowledgements}
The first author wants to thank Indiana University for hospitality and the Fulbright Foundation for a research grant that made this visit possible.

Part of this work was written while the second author was a
visitor at the Max Planck Institute for Mathematics in Bonn, which he would like to thank for its hospitality.  The second author also gratefully acknowledges an AMS-Simons travel grant which aided his travel to Bonn.

\section{Review of Morse theory for manifolds with boundary}\label{section:review-morse-theory-boundary}

\subsection{The absolute case}
We begin by recalling various definitions from \cite{BNR1} for absolute cobordisms, i.e.\  we are not considering
cobordism of embedded submanifolds yet.

\begin{definition}
An $(n+1)$-dimensional \emph{cobordism of manifolds with boundary} is a triple $(\O,\S_0,\S_1)$, where $\Omega$ is a manifold
with boundary, $\dim\O=n+1$, $\S_0$ and $\S_1$ are codimension 0 submanifolds of $\p\O$ and
\begin{itemize}
\item[$\bullet$] the boundary of $\O$ decomposes as a union $\p\O=\S_0\cup Y\cup\S_1$ for some $n$-dimensional manifold with boundary $Y$;
\item[$\bullet$] we have $\p\S_0=\S_0\cap Y=:N_0$, $\p\S_1=\S_1\cap Y=:N_1$ and $\p Y=N_0\cup N_1$. In other words, $Y$ is a cobordism between
$N_0$ and $N_1$.
\end{itemize}
\end{definition}

A schematic diagram is shown in Figure~\ref{figure:cylinder}.

\begin{figure}[H]
      \labellist
      \small\hair 0mm
      \pinlabel {$\O$} at 240 110
      \pinlabel {$\S_0$} at 40 110
      \pinlabel {$\S_1$} at 440 110
      \pinlabel {$N_0$} at 0 15
      \pinlabel {$N_1$} at 480 15
      \pinlabel {$Y$} at  240 15
      \endlabellist
      \begin{center}
      \includegraphics[scale=.45]{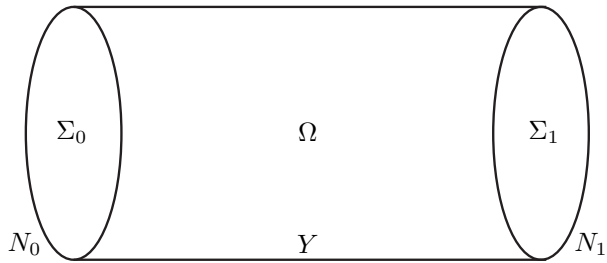}
      \end{center}
      \caption{An embedded cobordism.}
      \label{figure:cylinder}
\end{figure}

In many cases it is useful to assume that $\O$ is manifold with corners, i.e.\ near $N_0\cup N_1$ it is locally modelled on $\R^2_{\geqslant 0}\times\R^{n-1}$.
Thus, tubular neighbourhoods of $Y$, $\S_0$ and $\S_1$ are diffeomorphic to $Y\times [0,1)$, $\S_0\times[0,1)$
and $\S_1\times[0,1)$ respectively.

\begin{definition}
A cobordism $(\O,\S_0,\S_1)$ is a \emph{right product} if $\O\cong\S_1\times[0,1]$. It is called a \emph{left product} if $\O\cong\S_0\times[0,1]$.
\end{definition}
We remark that neither right nor left product cobordisms are necessary trivial.
We now recall a definition from \cite{KM}.

\begin{definition}\label{def:morse}
Let $(\O,\S_0,\S_1)$ be a cobordism. A function $f\colon\O\to[0,1]$ is called a \emph{Morse function} if $f^{-1}(0)=\S_0$, $f^{-1}(1)=\S_1$,
all the critical points of $f$ which lie in the interior $\Int\O$ are nondegenerate (the Hessian matrix is non-singular),
$f_Y:=f|_Y$ is a Morse function on $Y$ and, furthermore, for each $y\in Y$ such that $df(y)\neq 0$
we have
\be\label{eq:kerdf}
\ker df(y)\not\subset T_y Y.
\ee
\end{definition}

\begin{remark}
Our notation differs from that of \cite{BNR1} in two places. First, we use small $f$ to denote the Morse function on $\O$ (the capital $F$ is reserved
for another function). Furthermore we write $N_0$ and $N_1$ instead of $M_0$ and $M_1$, since $M$ is reserved for the \emph{membrane}: see
Definition~\ref{def:membrane} below.
\end{remark}

In \cite{KM} and \cite{BNR1}, condition (\ref{eq:kerdf}) was replaced by a condition on the gradient of $f$, so apparently depending on a
choice of a metric. The definition that we use does not involve choosing a Riemannian metric.  The next lemma shows that the two points of view are interchangeable.

\begin{lemma}[see \expandafter{\cite[Lemma 1.7]{BNR1}}]\label{lem:tang}
If $f$ satisfies \eqref{eq:kerdf} at each $y\in Y$, then there exists a Riemannian metric on $\O$ such that
$\nabla f$ is everywhere tangent to $Y$.
\end{lemma}

From now on, given a Morse function $M$ we shall assume that there is a Riemannian metric chosen so that $\nabla f$ is everywhere tangent to $Y$.

Let us now consider a critical point $z\in\O$ of $f$. If $z\in\Int\O$, we shall call it an \emph{interior} critical point. If $z\in Y$, then we
shall call it a \emph{boundary} critical point. Obviously, a boundary critical point is also a critical point of the restriction $f|_Y$.
There are two types of boundary critical points. The classification depends on whether the flow of $\nabla f$ near $z$ repels or attracts
in the direction of the normal to the boundary. More precisely, we have the following definition.

\begin{definition}\label{def:stable-unstabl}
A boundary critical point $z$ is called \emph{boundary stable}
if $T_zY$ contains the tangent space at $z$ to the unstable manifold of $z$.
If $T_zY$ contains the tangent space at $z$ to the stable manifold of $z$,
then $z$ is called \emph{boundary unstable}.
\end{definition}

\begin{definition}
The \emph{index} of a boundary critical point $z$ is the dimension of the stable manifold of $z$.
\end{definition}

\begin{remark}\label{remark:dbend-index-shift}
If $z\in Y$ is a boundary stable critical point of index $k$, then it is a critical point of $f|_Y$, but of index $k-1$.
\end{remark}

The change in the topology of the set $f^{-1}([0,x])$ as $x$ crosses the critical value corresponding to a boundary critical point
is described in detail in \cite[Section 2.4]{BNR1}. See also \cite{Hajduk,Blo}.

\begin{proposition}\label{prop:productcobordism}
Let $(\O,\S_0,\S_1)$ be a cobordism. Assume that there exists a Morse function $f\colon \O \to [0,1]$ such that $f$ has only
boundary stable (respectively, boundary unstable) critical points. Then $(\O,\S_0,\S_1)$ is a left product cobordism (respectively,
right product cobordism).
\end{proposition}

\begin{definition}\label{def:elementaryindex}
A cobordism $(\O,\S_0,\S_1)$ is called an \emph{elementary index $i$ right product cobordism}, (respectively an \emph{elementary index $i$~left
product cobordism}) if there exists a Morse function $f\colon\O\to[0,1]$, that has a single critical point, this critical point has index~$i$
and is boundary unstable (respectively boundary stable).
\end{definition}

The previous definition is the analogue in the boundary case of a single index~$i$ handle attachment corresponding to an interior critical point. The topological meaning is explained in detail in~\cite[Section 2]{BNR1} and in~\cite{Hajduk}.  We also refer to \cite{BNR2} for the study of homological properties of elementary right/left product cobordisms.

\subsection{The embedded case}
We shall now set up the notation for the embedded case.

\begin{definition}
A quadruple $(Z,\O,\S_0,\S_1)$ is called an \emph{embedded cobordism} if $Z$ is a closed manifold,
$(\O,\S_0,\S_1)$ is a cobordism and $\O$ is embedded in $Z\times[0,1]$
in such a way that $\O\cap Z\times\{0\}=\S_0$ and $\O\cap Z\times\{1\}=\S_1$ and $\O$ is transverse to $Z\times\{0,1\}$.
The \emph{codimension} of the embedding is the quantity $$\dim Z - \dim\S_0 =
\dim Z -\dim\O + 1.$$ An embedded cobordism is \emph{nondegenerate} if the function $F\colon Z\times[0,1]\to [0,1]$ given by projection
onto the second factor restricts to a Morse function on $\O$.
\end{definition}

Every embedded cobordism can be modified by a $C^2$-small perturbation to a nondegenerate one.
From now on we shall assume that all embedded cobordisms are nondegenerate.

Moreover, whenever we write $(Z,\O)$ as a cobordism, we understand the whole structure: the cobordism $\O$ is actually $(\O,\S_0,\S_1)$,
$Y=\cl(\p\O\sm(\S_0\cup\S_1))$, $F\colon Z\times[0,1]\to[0,1]$ is a function without critical points and $f=F|_\O$. We will usually denote
$m=\dim Z$ and $n+1=\dim\O$.

\begin{remark}
The whole theory could be developed in a more general setting, with $Z\times[0,1]$ replaced by a compact $(m+1)$-dimensional manifold $X$, and $F$ also allowed to have Morse critical points in $X\sm\O$. The special case $X=Z\times[0,1]$ is more transparent, and making the generalisation is
straightforward.
\end{remark}

\begin{definition}
Let $(Z,\O)$ be an embedded cobordism. The function $f\colon\O\to[0,1]$ given by $f=F|_\O$ is called the \emph{underlying Morse function}.
\end{definition}

We shall need one more notion.

\begin{definition}\label{def:regularhomotopy}
For a cobordism $(Z,\O)$ the family $F_t$ of functions from $Z\times[0,1]\to[0,1]$ is called a \emph{(nondegenerate) homotopy through submersions} if:
\begin{itemize}
  \item[(i)] for any $t$, $F_t(Z\times\{0\})=0$, $F_t(Z\times\{1\})=1$ and $F_t$ does not have any critical points on $Z\times[0,1]$;
  \item[(ii)] there exists a (possibly empty) finite set $\{t_1,\ldots,t_k\}\subset(0,1)$ such that if $t\not\in\{t_1,\ldots,t_k\}$, then $F_t$ restricts to a Morse function on $\O$.
\end{itemize}
\end{definition}

We will now choose a special metric on $Z\times[0,1]$, which will be used later.

\begin{lemma}\label{lem:funnylittlelemma}
For any two open subsets $V\subset U\subset Z\times[0,1]$, such that $V$ contains all the critical points of $f$ and $\ol{V}\subset U$, and for
any Riemannian metric $g$ on $U$ there exists a
Riemannian metric $h$ on $Z\times[0,1]$ such that we have
\[\nabla f(x)=\nabla F(x) \; \text{ for any } x \in \O\sm U.\]
Here we understand that $\nabla f$ is computed using the metric on $\O$ induced from the metric on $Z\times[0,1]$ (recall that $f=F|_\O$).
Furthermore, if $y\in Y$, we may assume that $\nabla f(y)\in T_yY$. Finally, the new metric agrees with $g$ on $V$.
\end{lemma}

\begin{proof}
We split the complement of $U$ into three cases: away from $\O$, in $\O \sm Y$ and finally in $Y$.

For each point $x\in Z\times[0,1]\setminus (U\cup\O)$ we choose a ball $B_x$ with centre $x$, disjoint from $\O\cup V$ and an arbitrary metric $h_x$ on $B_x$.

Now we consider $x\in \O\sm (U\cup Y)$. Let $m:= \dim Z + 1$.  Choose a ball $B_x$ which is disjoint from $V\cup Y$,
such that there is a local coordinate system $(x_1,\dots,x_k,x_{k+1},\dots,x_m)$ in $B_x$,
centred on $x$, in which
$\O\cap B_x=\{x_{k+1}=\dots=x_m=0\}$. As $B_x$ does not contain any critical points of $f$, we may assume that $\frac{\p F}{\p x_1}(z)\neq 0$ for all $z\in B_x$; we potentially need to relabel the coordinate system (and/)or  take a smaller ball for $B_x$.

Let us assume for now that $\frac{\p F}{\p x_1}(z)>0$ for all $z\in B_x$.
Choose an $m \times m$ matrix $M(z)$ such that $M_{1j}(z)=M_{j1}(z)=\frac{\p F}{\p x_j}(z)$ for $j=1,\dots, m$, with all other entries chosen so that $M(z)$ is positive definite and depends smoothly on $z$. This can be always achieved, since $M_{11}>0$. The matrix $M$ determines a metric on $B_x$, with respect to the standard basis $\{\p/\p x_i\, |\, i=1,\dots, m\}$ for $T_z Z$ induced by the coordinate system.  This metric is denoted $h_x$, and using this metric we have $\nabla F=(1,0,\dots,0)\in T_z\O$.
(If $\frac{\p F}{\p z}<0$ we get a metric such that $\nabla F=(-1,0,\dots,0)$.)  To see this recall that $\nabla F$ is the unique vector field $v$ such that $h_x(v_z,\cdot) = dF_z \in T^*_z Z$. Then observe that for any tangent vector $\begin{bmatrix}a_1 & \dots &  a_m \end{bmatrix}^T \in T_z Z$ we have:
\[\begin{bmatrix} 1 & 0 & \dots & 0 \end{bmatrix} \begin{bmatrix} M_{11}& \dots & M_{1m} \\ \vdots & \ast & \ast  \\ M_{m1} & \ast & \ast \end{bmatrix}\begin{bmatrix}a_1 \\ \vdots \\ a_m \end{bmatrix} = \sum_{i=1}^{m} a_i \frac{\p F}{\p x_i} = dF \begin{bmatrix} a_1 \\ \vdots \\ a_m \end{bmatrix},\]
where we have omitted the point $z \in Z$ from the notation.
By functoriality, $\nabla f$ is the image of the orthogonal projection of $\nabla F$ onto $T_z\O$. In particular, for any $z\in B_x\cap\O$, we have
$\nabla f(z)=\nabla F(z)$.

If $x\in Y\sm U$, we use the same argument to construct a metric $h_x$ on a ball $B_x$, where we assume that $B_x\cap V=\emptyset$.
This time we assume additionally that the coordinate system on $B_x$ is chosen so that $Y=\O\cap\{x_k=0\}$. Then we have $\frac{\p F}{\p x_j}(0,\dots,0)\neq 0$ for some
$j\in\{1,\dots,k-1\}$, for otherwise either $y$ is a critical point of $f$, or the condition \eqref{eq:kerdf} is violated. Then we choose a matrix $M(z)$
on (an again potentially smaller) $B_x$, which in turn induces a metric $h_x$ on $B_x$, similarly to the case above, in such a way that for all $z\in B_x\cap\O$ we have $\nabla f(z)=\nabla F(z)$ and, if additionally $z\in Y$, then $\nabla f(z)\in T_zY$.

The balls $B_x$, where $x$ runs through all points in $Z\times[0,1]\sm U$, together with $U$ constitute an open covering of $Z\times[0,1]$, which is compact. Let $\phi_U$, $\phi_x$
for $x\in Z\times[0,1]\sm U$ be a partition of unity
subordinate to this covering. We define a metric $h=\phi_U\cdot g+\sum_x \phi_x\cdot h_x$. It has all the desired properties.
\end{proof}

We conclude the section with a standard but fundamental result.

\begin{lemma}[Embedded Morse lemma]\label{lem:eml}
Assume that $p\in\O$ is a critical point of $f$ of index $k$. Then there exist local coordinates $x_1,\dots,x_{n+1},y_1,\dots,y_{m-n}$
in $Z\times[0,1]$ centred at $p$, such that in these coordinates:
\[F(x_1,\dots,x_{n+1},y_1,\dots,y_{m-n}) = -x_1^2-x_2^2-\dots-x_k^2+x_{k+1}^2+\dots+x_{n+1}^2+y_1+F(p),\]
and such that moreover:
\begin{itemize}
\item[$\bullet$] if $p$ is an interior critical point, then the intersection of $\O$ with this coordinate system is given by $$\{y_1=\dots=y_{m-n}=0\};$$
\item[$\bullet$] if $p$ is a boundary stable critical point, then $\O$ is given by
  \begin{equation}
\{y_1=\dots=y_{m-n}=0\}\cap\{x_1\geqslant 0\};\label{eq:boundx1}
\end{equation}
\item[$\bullet$] if $p$ is a boundary unstable critical point, then $\O$ is given by
 \begin{equation}
\{y_1=\dots=y_{m-n}=0\}\cap\{x_{n+1}\geqslant 0\}.\label{eq:boundxj}
\end{equation}

\end{itemize}
\end{lemma}
\begin{proof}
First let us consider the case that $p$ is an interior critical point. The Morse lemma (see e.g.\  \cite[Lemma 2.2]{Mi1})
says that there exist local coordinates $x_1,\dots,x_{n+1}$
on an open neighbourhood $V\subset\O$ such that $f=F|_\O$ is equal to
$F(0)-x_1^2-x_2^2-\dots-x_k^2+x_{k+1}^2+\dots+x_{n+1}^2$. Let $\wt{y}_1,\dots,\wt{y}_{m-n}$ be local coordinates in the normal bundle of $V$ in $Z\times[0,1]$. Then $x_1,\dots,x_{n+1},\wt{y}_1,\dots,\wt{y}_{m-n}$
form a local coordinate system in $Z\times[0,1]$ in a neighbourhood of $p$. By a linear change we may assume that $\frac{\p F}{\p \wt{y}_j}=0$
if $j>1$.
Let us define
\[y_1 := F(x_1,\dots,x_{n+1},\wt{y}_1,\dots,\wt{y}_{m-n})-(-x_1^2-x_2^2-\dots+x_{k+1}^2+\dots
+x_{n+1}^2)-F(p)\]
and $y_j:= \wt{y}_j$ for $j=2,\dots,m-n$.
Consider now the map of open neighbourhoods of $\R^{m+1}$ given by
\[\Phi\colon(x_1,\dots,x_{k+1},\wt{y}_1,\dots,\wt{y}_{m-n})\mapsto(x_1,\dots,x_{k+1},y_1,y_2,\dots,y_{m-n}).\]
The derivative of $\Phi$ at $0$ is diagonal and nondegenerate, hence $\Phi$ is a local diffeomorphism by the inverse function theorem.
Furthermore, the set $\{y_1=\dots=y_{m-n}=0\}$ is invariant under $\Phi$. The import of these two facts is that $(x_1,\dots,y_{m-n})$ forms
a local coordinate system near $p\in Z\times[0,1]$.  In this system $\O$ is given by $\{y_1=\dots=y_{m-n}=0\}$ and $F$ has the form as described in
the statement.

The proof in the case of boundary critical points is analogous. Using \cite[Lemma~2.6]{BNR1} we find local coordinates on $\O$
such that \eqref{eq:boundx1} or \eqref{eq:boundxj} is satisfied. Then we extend this coordinate system as in the case of an interior critical point.
We leave filling in further details.
\end{proof}

\section{Embedded gradient-like vector fields}\label{section:embedded-gl-vector-fields}

We need to develop the theory of embedded gradient-like vector fields in order to prove our subsequent results on rearrangement, cancellation and splitting.

We will soon introduce a notion of gradient-like vector fields for embedded submanifolds.
First let us recall the definition of a gradient-like vector field for manifolds with boundary.

\begin{definition}[Compare \expandafter{\cite[Definition 1.5]{BNR1}}]
Let $(\O,\S_0,\S_1)$ be a cobordism, and let $f$ be a Morse function.  We shall say that a vector field~$\xi$ on $\O$ is
\emph{gradient-like with respect to $f$}, if the following
conditions are satisfied:
\begin{itemize}
\item[(a)] $\xi\cdot f := df(\xi) >0$ away from the set of critical points of $f$;
\item[(b)] if $p$ is a critical point of $f$ of index $k$, then there exist local coordinates $x_1,\dots,x_{n+1}$
in a neighbourhood of $p$, such that
\[f(x_1,\dots,x_{n+1})=f(p)-(x_1^2+\dots+x_k^2)+(x_{k+1}^2+\dots+x_{n+1}^2)\]
and
\[\xi=(-x_1,\dots,-x_k,x_{k+1},\dots,x_{n+1})\text{ in these coordinates}.\]
\item[(b')] Furthermore, if $p$ is a boundary critical point, then the above coordinate system can be chosen so that
$Y=\{x_j=0\}$ and $\O=\{x_j\geqslant 0\}$  for some $j\in\{1,\dots,n+1\}$.
\item[(c)] The vector field $\xi$ is tangent to $Y$ on all of $Y$.
\end{itemize}
\end{definition}

Now let us define an analogue for embedded cobordisms. Observe that we cannot simultaneously assume that $\xi\cdot F>0$ away from the critical points of $F$ and that $\xi$ is everywhere tangent to $\O$, because these two conditions are mutually exclusive if $f=F|_\O$ has critical points. The vector
field that we define below has a critical point at each critical point of $f$. The following definition comes from R.~Sharpe's paper \cite[page 67]{Sh}
and turns out to provide a very useful analytic language for studying embedded cobordisms.

\begin{definition}\label{def:eglv}
Let $(Z,\O)$ be an embedded cobordism, with $F\colon Z\times[0,1]\to[0,1]$ the projection and the underlying Morse function $f$.
A vector field $\xi$ on $Z \times [0,1]$ is an \emph{embedded gradient-like vector field} with respect to $F$ if:
\begin{itemize}
\item[(a)] for any $x\in Z\times[0,1]$ which is not a critical point of $f$, we have $(\xi\cdot F)_x = dF_x(\xi_x)>0$;
\item[(b)] for any $x\in\O$ we have $\xi_x\in T_x(\Omega)$, and for any $y\in Y$ we have $\xi_y\in T_y Y$;
\item[(c)] for any $p\in\O$ such that $df(p)=0$, there exists an open subset $U\subset Z\times[0,1]$ with a choice of
local coordinates $(x_1,\dots,x_{n+1},y_1,\dots,y_{m-n})$ centred at $p$
such that $U\cap\O$ is given by $\{y_1=\dots=y_{m-n}=0\}$ (if $p$ is a boundary critical point, than $U\cap\O=\{x_1\geqslant 0,y_1=\ldots=y_{m-n}=0\}$
and $Y=\{x_1=y_1=\ldots=y_{m-n}=0\}$), $\xi$ in these local coordinates has the form
\begin{equation}\label{eq:c1}
(-x_1,-x_2,\dots,-x_k,x_{k+1},\dots,x_{n+1},(y_1^2+\dots+y_{m-n}^2),0,\dots,0)
\end{equation}
and
\begin{equation}\label{eq:c2}
F(x_1,\dots,y_{m-n})=F(p)-x_1^2-\dots-x_k^2+x_{k+1}^2+\dots+x_{n+1}^2+y_1.
\end{equation}
\end{itemize}
\end{definition}

Note that
\[dF(\xi)=\sum_{i=1}^{n+1}2x_i^2+\sum_{\ell=1}^{m-n}y_{\ell}^2\]
so that (a) and (c) are consistent.

We have the following result.
\begin{proposition}
For any cobordism $(Z,\O)$ there exists an embedded gradient-like vector field.
\end{proposition}

For the convenience of the reader we present a straightforward proof.

\begin{proof}
Let $p_1,\dots,p_r$ be the critical points of $F|_\O$. For each $j\in\{1,\dots,r\}$ we choose an open neighbourhood $U_j'$ in $Z \times[0,1]$ such that the Embedded Morse Lemma~\ref{lem:eml} holds; i.e.\ there exist coordinates of a special form as set out in that lemma.  We choose $U_j$ to be a neighbourhood of $p_j$ such that $\ol{U_j}\subset U_j'$. Furthermore, let $V$ be
an open subset of $Z\times[0,1]$ such that $V\cup\bigcup U_j'=Z\times[0,1]$ and $U_j\cap V=\emptyset$ for any $j\in\{1,\dots,r\}$.
Let $\phi_V,\phi_1,\dots,\phi_r$ be a partition of unity subordinate to the covering $V\cup\bigcup U_j'$. In particular $\phi_j|_{U_j}\equiv 1$.

Choose a Riemannian metric on any $U_j$ such that the local coordinates $x_1,\dots,x_{n+1},$ $y_1,\dots,y_{m-n}$ are orthogonal.
By Lemma~\ref{lem:funnylittlelemma} there exists a Riemannian metric on $Z\times[0,1]$ such that $\nabla F(x)$ is tangent to $\O$ for all $x\in\O\sm V$
and $\nabla F(x)$ is tangent to $Y$ for all $x\in Y\sm V$.

We define a vector field on $V$ by $\xi_V=\nabla F$ and $\xi_j$ on $U_j'$ by the explicit formula \eqref{eq:c1}. Then $\xi=\phi_V\xi_V+\sum \phi_j\xi_j$
is a vector field on $Z\times[0,1]$ which, by construction, satisfies the desired properties.
\end{proof}

We remark that $\xi$ has a critical point at each critical point $p$ of $f$. This is not a Morse critical point, because
the coordinate corresponding to $y_1$ vanishes up to order $2$. Nevertheless, we have well-defined stable and unstable manifolds of $\xi$ at $p$, which we now discuss. The following lemma is a consequence of the local description of a critical point.

\begin{lemma}
In a neighbourhood of a critical point, in local coordinates as in Definition~\ref{def:eglv}, the stable manifold is given by
\[\{x_{k+1}=\dots=x_{n+1}=0,\ y_2=\dots=y_{m-n}=0,\ y_1<0\}.\]
The unstable manifold is given by
\[\{x_1=\dots=x_k=0,\ y_2=\dots=y_{m-n}=0,\ y_1>0\}.\]
The intersection of a stable manifold $($respectively: unstable manifold$)$ with a level set $F^{-1}(p-\e)$, for $\e>0$ sufficiently small $($respectively: with a level set $F^{-1}(p+\e))$ is a $k$-dimensional disc $($respectively: $(n+1-k)$-dimensional disc$)$. The boundary of the disc is the stable manifold
of $\xi|_\O$ $($respectively: the unstable manifold of $\xi|_\O)$.
\end{lemma}

The following terminology is essentially due to B.~Perron \cite{Pe}.

\begin{definition}\label{def:membrane}
Let $p$ be a critical point of $f$. The \emph{ascending membrane} $M^u_p$ is the unstable manifold of $p$ with respect to
$\xi$. The \emph{descending membrane} $M^s_p$ is
the stable manifold of $p$ with respect to $\xi$.
\end{definition}

From now on, when speaking of a stable and unstable manifold of $\xi$, we will understand a stable and unstable manifold of $\xi|_\Omega$, since we will use the term membrane of Definition~\ref{def:membrane} for the ambient version.

\subsection{Integrating the vector field to recover the Morse function}\label{section:recovering-morse-function}

Starting with a Morse function $F$ and a gradient-like vector field $\xi$, we might wish to alter the vector field $\xi$ to $\xi'$, and the altered vector field $\xi'$ may then not necessarily be a gradient-like vector field for $F$. Under some conditions we shall be able to find a function $F'$, such that $\xi'$ is a gradient-like vector field with respect to $F'$. This is the idea of the
Vector Field Integration Lemma~\ref{lem:integrate} below.

We remark that in this paper Integration Lemma~\ref{lem:integrate} is only used in the proof of Elementary Cancellation Theorem~\ref{lem:eec}.  We present the Integration Lemma separately since we think it is of interest independently from the cancellation theorem.

Before stating and proving the Vector Field Integration Lemma, first we need to introduce some more terminology. Part (a) of the next definition is standard.

\begin{definition}\label{def:trajectories}
Let $\xi$ be a smooth vector field $Z\times[0,1]$.
\begin{itemize}
\item[(a)] A \emph{trajectory} is a map $\gamma\colon A\to Z\times[0,1]$, where $A$ is a connected subset of $\R$,
such that $\frac{d}{dt}\gamma(t)=\xi(\gamma(t))$. We will always assume that $A$ is maximal, i.e.\ $\gamma$ cannot be extended over a larger subset of $\R$. Note that, up to reparametrisation, that is changing $t$ to $t+a$ for some $a\in\R$,
exactly one trajectory of $\xi$ passes through a given point of $Z\times[0,1]$.
\item[(b)] A \emph{broken trajectory} is a union of trajectories $\gamma_1,\dots,\gamma_s$ such that for any $j=1,\dots,s-1$ we have
$\lim_{t\to\infty}\gamma_j(t)=\lim_{t\to-\infty}\gamma_{j+1}(t)$. These limits are critical points of $\xi$.
\end{itemize}
\end{definition}

Note that there is no mention of the function $F$ in the next definition, which introduces the notion of an almost gradient-like vector field.  In Lemma~\ref{lem:integrate} we will see that an almost gradient-like vector field with one extra assumption is sufficient to be able to recover a Morse function with respect to which the vector field is an embedded gradient-like vector field.

\begin{definition}\label{def:AG}
A vector field $\xi$ on $Z\times[0,1]$ is called \emph{almost gradient-like} if it satisfies the following conditions:
\begin{itemize}
\item[(AG1)] $\xi$ is tangent to $\O$ and to $Y$;
\item[(AG2)] $\xi$ has no critical points on $Z\times[0,1]\sm\O$ and no critical points on $Z\times\{0,1\}$;
\item[(AG3)] $\xi$ has finitely many critical points on $\O$. For each critical point $z$ of $\xi$, there exist local coordinates near $z$, denoted
by $x_1,\dots,x_{n+1},y_1,\dots,y_{m-n}$ such that $\O$ in the local coordinates is given by $y_1=\dots=y_{m-n}=0$ (if $z$ belongs to $Y$, there is
one more equation, namely $x_1\geqslant 0$) and there is an index $k$ such that in these local coordinates $\xi$ has the form as in \eqref{eq:c1}:
\[(-x_1,-x_2,\dots,-x_k,x_{k+1},\dots,x_{n+1},(y_1^2+\dots+y_{m-n}^2),0,\dots,0);\text{ and}\]
\item[(AG4)] if a trajectory $\gamma(t)$ does not hit $Z\times\{0\}$, then $\lim_{t\to-\infty}\gamma(t)$ exists (and by standard arguments observing that $\lim_{t\to-\infty}\gamma'(t)=0$ it is a critical point of $\xi$). Likewise, if $\gamma(t)$ does not hit $Z\times\{1\}$, then $\lim_{t\to+\infty}\gamma(t)$ exists.
\end{itemize}
\end{definition}

If $\xi$ is an embedded gradient-like vector field, then it is easy to see that it is an almost gradient-like vector field. Given an almost gradient-like vector field $\xi$ it is possible that there does not exist any function $F$ with respect to which $\xi$ is an embedded gradient-like vector field. For example, the conditions (AG1)--(AG4) do not exclude the possibility that there exist a broken trajectory starting and ending at the same point. As next result shows, this is the only obstruction.

\begin{lemma}[Vector Field Integration Lemma]\label{lem:integrate}
Suppose that $\xi$ is an almost gradient-like vector field such that there are no broken trajectories starting and ending at the same point.
Then there exists a smooth function $F$ without critical points on $Z\times[0,1]$ such that $\xi$ is an embedded gradient-like vector field with respect to~$F$.
\end{lemma}
\begin{proof}
Let $z_1,\dots,z_s$ be the critical points of $\xi$. We introduce a partial order relation on the critical points, namely we say that $z_i<z_j$ if there is at least one broken trajectory starting at $z_i$ and ending at $z_j$. The assumption of the lemma guarantees that this is a partial order. Let us relabel the critical points so that
if $i<j$, then we cannot have $z_j<z_i$. This relabelling is, in general, not unique.  The proof of Lemma~\ref{lem:integrate} continues after the statement and proof of Lemma~\ref{lem:csn}.

\begin{lemma}\label{lem:csn}
There exist open neighbourhoods $V_1,\dots,V_s$ of $z_1,\dots,z_s$ such that for any $i\leqslant j$, there is no trajectory which leaves $V_j$ and then enters $V_i$.
\end{lemma}
\begin{proof}[Proof of Lemma~\ref{lem:csn}]
The proof follows the ideas of \cite[Assertion 1, page 50]{Mi2}. Suppose that the statement is false.  That is, suppose that for all open neighbourhoods $V_1,\dots,V_s$ there exists $i \leqslant j$ and a trajectory which leaves $V_j$ and later enters $V_i$.

Let $U^i_0$ be a coordinate neighbourhood of $z_i$ from (AG3) of Definition~\ref{def:AG}. For $r\geqslant 1$ we
define $U^i_r=\{x_1^2+\dots+x_{n+1}^2+y_1^2+\dots+y_{m-n}^2\leqslant \frac{\e}{r}\}$ for some $\e$ sufficiently small (so that $U^i_1\subset U^i_0$).

Now suppose that there are indices $i,j$, with $i\leqslant j$ and, for all $r$, that there is a trajectory $\gamma_r$ going from $U_r^j$ to $U_r^i$. Let us choose a maximal $j$ such
that this holds. Working in local coordinates
we convince ourselves that $\gamma_r$ must intersect $\p U_1^j$ and $\p U_1^i$. Let $w_r^1$ be a point where $\gamma_r$ leaves $\p U_r^j$, and let $w_r^2$ be the point where
$\gamma_r$ hits $\p U^j_1$ after leaving $\p U_r^j$. Let $w_r^3$ be the point where $\gamma_r$ hits $\p U_r^i$ for the first time after $w_r^2$.

Since all points $w_r^2$, $r=2,3,\ldots$ belong to the sphere $\p U^j_1$, up to passing to a subsequence we can assume
$w_r^2$ converges to a point $w_0\in \p U^j_1$.  Let $\gamma_0$ be the trajectory through $w_0$. We claim that $\lim_{t\to-\infty}\gamma_0(t)=z_j$.
To see this, observe that for any $l>r$ the trajectory through $w_l^2$ hits $\p U^j_r$ in the past. As $w_l^2\to w_0$ and $\p U^j_r$ is closed, we
infer that $\gamma_0$ hits $\p U^j_r$ in the past as well.  But $r$ was arbitrary, so there exists a sequence $t_k$
converging to $-\infty$, such that $\gamma_0(t_k)\to z_j$. By (AG4), $\gamma_0$ is a trajectory starting from $z_j$. We do not
claim that $\gamma_0$ ends in $z_i$, because the sequence of trajectories $\gamma_r$ can converge to a broken trajectory, a part of which
is $\gamma_0$.

To complete the proof, let us look at $\lim_{t\to+\infty}\gamma_0(t)$. Observe that the time which the trajectory $\gamma_r$ takes to go from $w_r^2$ to $w_r^3$
goes to infinity as $r\to\infty$ (because the speed of the vector field near critical points is very small).
Therefore, $\gamma_0(t)$ exists for any $t>0$; in particular $\gamma_0$ cannot hit $Z\times\{1\}$. By (AG4),
$\lim_{t \to +\infty}\gamma_0(t)=z_k$ for some $k\in\{1,\dots,s\}$. This means that $k>j$.

Consider now neighbourhoods $U^k_r$ as defined above. As $\lim\gamma_0(t)=z_k$, it follows that $\gamma_0$ hits all the boundaries $\p U^k_r$. If we fix $r$,
it follows that for $l$ sufficiently large $\gamma_l$ hits $\p U^k_r$ as well. We can now relabel the trajectories so that $\gamma_r$ hits
$\p U_r^k$ and then $U_r^i$ i.e.\ we want a trajectory which hits $\p U_p^k$ to have index $p$; this may involve passing to a further subsequence in the $\gamma_r$.
Then the sequence $\{\gamma_r\}_{r \geqslant 1}$ may be considered as a sequence of trajectories coming close to $z_k$ first, and then to $z_i$.
But we assumed that $j$ is the maximal index for which this is possible, and $k>j$, so we obtain a contradiction.
\end{proof}

We resume the proof of Lemma~\ref{lem:integrate}.  Let us now choose open neighbourhoods $V_1,\dots,V_s$ of $z_1,\dots,z_s$ as given to us by Lemma~\ref{lem:csn}. We define a function $F$ to be
equal to $-x_1^2-\dots-x_k^2+x_{k+1}^2+\dots+x_{n+1}^2+y_1+j/(s+1)$ on $\ol{V}_j$, where the signs agree with the signs in (AG3). In particular, $F(z_j)=j/(s+1)$.
For the sake of completeness we define $\ol{V}_0:= Z\times\{0\}$ and $\ol{V}_{s+1}:= Z\times\{1\}$. We also put $F(\ol{V}_0)=0$ and $F(\ol{V}_{s+1})=1$.

Now let us choose any point $z\in Z\times[0,1]\sm\bigcup\ol{V}_j$. Let $\gamma$ be a trajectory through $z$.  Reparametrise $\gamma$ so that $\gamma(0)=z$.
By (AG4) there exists $a,b\in\R$, $-a<0<b$ such that $\gamma(-a)\in \ol{V}_i$, $\gamma(b)\in \ol{V}_j$ for some $i,j$,
and $\gamma(-a,b)$ does not intersect $V_1\cup\dots\cup V_s$. By Lemma~\ref{lem:csn} we have $i<j$. Then we define
\[F(z)=\frac{b}{a+b}F(\gamma(-a))+\frac{a}{a+b}F(\gamma(b)).\]
Since $F(\gamma(-a))<F(\gamma(b))$ (because $i<j$), the function $F$ increases along $\gamma$. Thus $\xi\cdot F(z)>0$.

In general, the function $F$ is a continuous, piecewise smooth function with all non-smooth points lying on $\p V_1\cup\dots\cup \p V_s$. The function $F$ satisfies $\xi\cdot F(z)>0$ for $z\in V_j\sm\{z_j\}$ and whenever $\xi\cdot F(z)$ is well-defined.
Therefore, we can slightly perturb $F$ to a smooth function, which still satisfies $\xi\cdot F>0$
away from $z_1,\dots,z_s$.
\end{proof}

\subsection{The embedded isotopy lemma}\label{section:isotopy-lemma}

Next we are going to sketch a proof of the embedded analogue of the Isotopy Lemma of~\cite[Lemma 4.7]{Mi2}.

\begin{lemma}[Embedded Isotopy Lemma]\label{lem:eil}
Let $(Z,\O)$ be an embedded cobordism. Suppose that there are two level sets $a,b\in[0,1]$ with $a<b$ such that $f$ has no critical points
on $\O\cap f^{-1}[a,b]$. Let $\xi$ be an embedded gradient-like vector field.
Assume additionally that there is a diffeomorphism $h$ of the triple $(Z,\O\cap f^{-1}(b),Y\cap f^{-1}(b))$ to itself which is isotopic to the identity.

Then there exists an embedded gradient-like vector field $\xi'$, agreeing with $\xi$ away from $F^{-1}(a,b)$,
such that $\psi'=h\circ\psi$, where $\psi$ and $\psi'$ are the diffeomorphisms
\[\psi,\psi' \colon (Z,\O\cap f^{-1}(a),Y\cap f^{-1}(a)) \to (Z,\O\cap f^{-1}(b),Y\cap f^{-1}(b)) \]
induced by the flows of the vector fields $\xi$ and $\xi'$ respectively.
\end{lemma}
\begin{proof}
We follow \cite[Proof~of~Lemma~4.7]{Mi2}. Let $h_t$, $t\in[a,b]$ be an isotopy of the triple $(Z,\O\cap f^{-1}(b),Y\cap f^{-1}(b))$ so that $h_a$ is the identity and $h_b=h$.
We also assume that $h_t$ does not depend on $t$ for $t$ close to $a$ and $b$. Let $H\colon Z\times[a,b]\to Z\times[a,b]$ be given by $(x,t)\to (h_t(x),t)$.
Define a diffeomorphism $\Psi\colon Z\times[a,b]\to Z\times[a,b]$ by integrating the flow of $\xi$; that is, if we take a point $(x,t)\in Z\times[a,b]$, there is a trajectory of $\xi$ passing through that point. This trajectory hits a point $(x',a)\in Z\times\{a\}$. We define $\Psi(x,t)=(x',t)$. It is easy to check
that $\Psi$ is a diffeomorphism. Now let $\Phi := \Psi^{-1}\circ H\circ\Psi\colon Z\times[a,b]\to Z\times[a,b]$ and we define $\xi':=\Phi_*(\xi)$ on $Z\times[a,b]$.  Since $H$ is the identity near $Z\times\{a,b\}$, we infer that $\Phi$ is the identity near $Z\times\{a,b\}$, and hence $\xi'$ agrees with $\xi$ in a neighbourhood of $Z\times\{a,b\}$.  We extend $\xi'$ to $Z\times[0,1]$ by making $\xi'$ equal to $\xi$ on $Z\times([0,a]\cup [b,1])$

 By definition, the flow of $\xi'$ induces a diffeomorphism of the triple $(Z\times\{a\},\O\cap f^{-1}(a),Y\cap f^{-1}(a))$
to $(Z,\O\cap f^{-1}(b),Y\cap f^{-1}(b))$, which is equal to $h\circ\psi$.
Note that, by construction, $\xi'$ is tangent to $\O$ at all points in $\O\cap f^{-1}[a,b]$ and is tangent to $Y$ on $Y\cap f^{-1}[a,b]$.
\end{proof}

\section{Rearrangement of critical points}\label{section:rearrangement}

The aim of this section is to prove the Elementary and Global Rearrangement Theorems, in Sections~\ref{section:elem-rearrangement-thm} and \ref{section:global-rearrangement} respectively.

\subsection{The embedded elementary rearrangement theorem}\label{section:elem-rearrangement-thm}

The rearrangement theorem in the embedded case is stated or proved
in many places, like \cite{GS,Pe,Ro,Sh}. For the convenience
of the reader, and because we would also like to deal with boundary critical points, we present a proof.

\begin{theorem}[Elementary Rearrangement Theorem]\label{thm:ert}
Let $(Z,\O,\S_0,\S_1)$ be an embedded cobordism and let~$f$ be the underlying Morse function. Suppose that $f$ has exactly two critical
points $z_1$ and $z_2$ with $f(z_1)<f(z_2)$. Let $\xi$ be an embedded gradient-like vector field.
For $i=1,2$, let $M^u_i\subset Z\times[0,1]$ (respectively: $M^s_i\subset Z\times[0,1]$),
be the ascending membrane of $z_i$ (respectively: the descending membrane)
under the flow of $\xi$.
If $$M^u_1\cap M^s_2=\emptyset,$$ then for any two values $a,b\in[0,1]$
there exists a function $G\colon Z\times[0,1]\to[0,1]$
such that
\begin{itemize}
\item[(E1)] $G$ has no critical points;
\item[(E2)] $G(z_1)=a$ and $G(z_2)=b$;
\item[(E3)] The restriction $g:=G|_\Omega$ is Morse. It has two critical points, $z_1$ and $z_2$ with the same type as $f$.
\end{itemize}
\end{theorem}

\begin{remark}
  We note that the new Morse function $G$ can be chosen so that there is a nondegenerate homotopy through submersions between $G$ and the old Morse function~$F$.
\end{remark}

\begin{proof}
The proof goes along similar lines to \cite[Section 4]{Mi2} (see also \cite[Proposition 4.1]{BNR1}).
We define $K_1=M^u_1\cup M^s_1$ and $K_2=M^u_2\cup M^s_2$. The emptiness of $M^u_1\cap M^s_2$ implies that $K_1\cap K_2=\emptyset$.
Let $T_1=K_1\cap Z\times\{0\}$ and $T_2=K_2\cap Z\times\{0\}$. We see that $T_1$ and $T_2$ are not empty, because $\dim M^s_1,\dim M^s_2\geqslant 1$.
Also $T_1\cap T_2=\emptyset$.

Let $W_1\supset T_1$ and $W_2\supset T_2$ be
two disjoint open subsets of $Z\times\{0\}$. Let $\mu\colon Z\times\{0\}\to[0,1]$ be
a smooth function, such that $\mu(W_1)=0$ and $\mu(W_2)=1$. We extend $\mu$ to $Z\times[0,1]$ as follows: if $x\in K_1$, we put $\mu(x)=0$;
if $x\in K_2$, we put $\mu(x)=1$.  If $x\not\in(K_1\cup K_2)$, then the trajectory of $\xi$ through $x$ hits $Z\times\{0\}$ in a unique
point $y\in Z\times\{0\}$.
Then we define $\mu(x):=\mu(y)$. This definition implies in particular that $\mu$ is constant along all the trajectories of $\xi$.

Following Milnor we choose a smooth function
$\Psi\colon [0,1]\times[0,1]\to[0,1]$ satisfying:
\begin{itemize}
\item[$\bullet$] $\frac{\p\Psi}{\p x}(x,y)>0$ for all $(x,y)\in[0,1]\times[0,1]$;
\item[$\bullet$] there exists $\delta>0$, such that $\Psi(x,y)=x$ for all $x\in[0,\delta]\cup[1-\delta,1]$ and $y\in [0,1]$;
\item[$\bullet$] for any $s\in(-\delta,\delta)$ we have $\Psi(f(z_1)+s,0)=a+s$ and
$\Psi(f(z_2)+s,1)=b+s$.
\end{itemize}

We now define
\[G(x):=\Psi(F(x),\mu(x))\textrm{ and }g:=G|_\O.\]
Observe that by the chain rule $$\xi\cdot G=dG(\xi) = \Big(\frac{\p \Psi}{\p x}dF+\frac{\p\Psi}{\p y}d\mu \Big)(\xi)= \frac{\p \Psi}{\p x}\xi\cdot F+\frac{\p\Psi}{\p y}\xi\cdot\mu.$$ Since $\mu$ is constant on the trajectories of $\xi$,
we have $\xi\cdot\mu=0$. As $\frac{\p\Psi}{\p x}>0$ and $\xi$ is an embedded gradient-like vector field, we see that $\xi\cdot G(x)\geqslant 0$ with equality if and only if $x$ is a critical point of $f$. On the other hand by the definition of $\Psi$ we have that, near a critical point of $f$,
the function $G$ is equal to $F$ plus a constant. Hence $G$ has no critical points. We compute that by construction $G(z_1)=a$ and $G(z_2)=b$:
\[G(z_1) = \Psi(F(z_1),\mu(z_1)) = \Psi(F(z_1),1) = \Psi(f(z_1),1) = a;\]
\[G(z_2) = \Psi(F(z_2),\mu(z_2)) = \Psi(F(z_2),1) = \Psi(f(z_2),1) = b.\]
Now, to check that (E3) is still satisfied, let $x\in\O$. We have $g(x)=\Psi(f(x),\mu(x))$. As $\mu$ is everywhere tangent to $\O$, we can repeat the above argument to show that
$\xi\cdot g\geqslant 0$, with equality only at the critical points of $f$. Since $g-f$ is constant near critical points, the types of
the critical points are the same.

Finally consider $x\in Y$, such that $x$ is not a critical point of $f$. Then $\xi$ is tangent to $T_xY$ by definition and $\xi\cdot g>0$ as $g$ is just the restriction of $G$. This means that $dg(\xi)>0$, and in particular that $T_xY\not\subset\ker dg$. Thus, as required, $g$ is a Morse function in the sense of Definition~\ref{def:morse}.
\end{proof}

It is easy to see that the argument of Theorem~\ref{thm:ert} can be repeated if $f$ has more critical points and suitable intersections of
stable/unstable manifolds are empty.  This is made precise in Theorem~\ref{thm:egrt} below.

\subsection{The embedded Morse--Smale condition}\label{section:morse-smale-condition}

In the following we write $M^{s}(z)$ (respectively $M^u(z)$), to denote the descending and ascending membranes of the critical point $z$.
We write $W^s(z)$ and $W^u(z)$ to denote the stable and the unstable manifolds of $\xi|_\Omega$, with $W^s(z),W^u(z)\subset\O$.
If $z$ is a boundary critical point, we denote by $W^s_Y(z)$ and $W^u_Y(z)$, respectively the stable and unstable manifold of the vector field $\xi$
restricted to $Y$.

\begin{definition}\label{def:emsc}
The vector field $\xi$ satisfies the \emph{embedded Morse--Smale conditions} if for any two critical points $z_1$ and $z_2$ of $f$, the intersections
of $M^s(z_1)$ with $M^u(z_2)$ are transverse in $Z\times[0,1]\sm\O$, the intersections of $W^s(z_1)$ and $W^u(z_2)$ are transverse
in $\O\sm Y$ and the intersections of $W^s_Y(z_1)$ with $W^u_Y(z_2)$ are transverse in $Y$.
\end{definition}

\begin{lemma}
For every embedded gradient-like vector field $\xi$ there exists a $C^2$-small perturbation $\xi'$ which satisfies
the embedded Morse--Smale condition.
\end{lemma}
\begin{proof}[Sketch of proof]
This is a standard result combining the fact that the transversality condition is open (see \cite[Section 29]{Arn})
together with Lemma~\ref{lem:eil}. We leave the details to the reader.
\end{proof}

We show, in table form, the dimensions of stable and unstable manifolds of a critical point. In Table~\ref{tab:one}, we
assume that $z$ is a critical point of $\O$ of index $k$, and we recall that $\dim \O = n+1$. Also recall that the index of a critical point $z$ is the dimension of its stable manifold $W^s(z)$.


\begin{table}[t]
\begin{center}
\small
\begin{tabular}[t]{lcccccc}
\toprule
Type of $z$ & $M^s\sm\O$ & $M^u\sm\O$ & $W^s\sm Y$ & $W^u\sm Y$ & $W^s_Y$ & $W^u_Y$ \\
\midrule
interior & $k+1$ & $n+2-k$ & $k$ & $n+1-k$ & $\emptyset$ & $\emptyset$ \\
bdy. stable & $k+1$ & $n+2-k$ & $k$ & $\emptyset$ & $k-1$ & $n+1-k$ \\
bdy. unstable & $k+1$ & $n+2-k$ & $\emptyset$ & $n+1-k$ & $k$ & $n-k$\\
\bottomrule
\end{tabular}%
\end{center}
\medskip
\caption{Dimensions of various stable and unstable manifolds. Here $\emptyset$ means that the corresponding manifold is empty, as does dimension $-1$.}\label{tab:one}
\end{table}

We remark that the intersection of the stable manifold of one point with the unstable manifold of another, unless empty, must have dimension at least one.
Therefore, the embedded Morse--Smale condition (Definition~\ref{def:emsc}) yields the following result.

\begin{proposition}\label{prop:ems}
Let $z$ and $w$ be two critical points of $f$ of indices $k$ and $l$ respectively. Let $m := \dim Z$. Suppose that $\xi$ satisfies the embedded Morse--Smale condition.  Then the
intersection $M^u(z)\cap M^s(w)$ is empty if at least one of the following conditions hold:
\begin{itemize}
\item[$\bullet$] $k=l$, $m\geqslant n+2$  and either $z$ is not a boundary stable critical point or $w$ is not
a boundary unstable critical point;
\item[$\bullet$] $k>l$ and $m\geqslant n+1$;
\item[$\bullet$] $z$ is an interior critical point, $w$ is boundary unstable and $l-k\leqslant m-n-2$;
\item[$\bullet$] $z$ is a boundary stable critical point, $w$ is interior and $l-k\leqslant m-n-2$.
\end{itemize}
\end{proposition}

\begin{proof}
In each case, the proof follows by checking that each of \[\dim (M^u(z)\sm\O) + \dim (M^s(w)\sm\O)  \leqslant m+1,\]
\[\dim (W^u(z)\sm Y) + \dim (W^s(w)\sm Y) \leqslant n+1\]
and
\[\dim W^u_Y(z) + \dim W^s_Y(w)  \leqslant n\]
are satisfied.
\end{proof}

\subsection{The embedded global rearrangement theorem}\label{section:global-rearrangement}

As a corollary of Proposition~\ref{prop:ems}, we obtain the following global rearrangement theorem. In codimension 2 or more, as in Theorem~\ref{thm:egrt}, this is the standard rearrangement theorem.  In codimension~1 the situation is more complicated and will be addressed in Theorem~\ref{prop:egrt1}.

The next definition was already given in the introduction; for the convenience of the reader we recall it here.

\begin{definition}\label{defn:admissible-config}
  A \emph{configuration} $\Xi$ of the critical points of an embedded cobordism $(Z,\O)$ is an assignment of a value $\Xi(z_i) \in (0,1)$ to each critical point $z_i$ of~$f$.

  An \emph{admissible configuration} is a configuration satisfying the following conditions:
\begin{itemize}
\item[(A1)] if $z,w$ are critical points with indices $k$, $l$ with $k<l$, then $\Xi(z)<\Xi(w)$;
\item[(A2)] if $z,w$ have the same index $k$ and if $z$ is boundary stable and $w$ is boundary unstable, then $\Xi(z)<\Xi(w)$.
\end{itemize}
\end{definition}

The Global Rearrangement Theorem says that any admissible configuration can be realised by changing the Morse function.

\begin{theorem}[Global Rearrangement Theorem]\label{thm:egrt}
Suppose $(Z,\O)$ is a cobordism.  Given an admissible configuration $\Xi$ of the critical points of $f$,
if $\codim(\O\subset Z\times[0,1])\geqslant 2$, there exists a function $G\colon Z\times[0,1]\to[0,1]$ without critical points, homotopic through submersions to $F$, which restricts
to a Morse function $g\colon\O\to[0,1]$, such that $g$ restricted to the critical points agrees with~$\Xi$ (the type and index of each critical point is preserved).
\end{theorem}

\begin{remark}
If $z,w$ have the same index $k$, both critical
points are boundary stable, both are boundary unstable or at least one of them is interior, then we can have $g(z)<g(w)$, $g(z)=g(w)$ or $g(z)>g(w)$, as we please, provided condition (A2) is satisfied.

For example, if we have three critical points $z$, $w$ and $v$ of index $k$, where $z$ is boundary stable
and $w$ is boundary unstable and $v$ is interior, then in general we cannot arrange that $g(z)>g(v)$ and $g(v)>g(w)$ simultaneously, since this would violate (A2).
\end{remark}

\begin{proof}[Proof of Theorem~\ref{thm:egrt}]
First apply Theorem~\ref{thm:ert} to arrange the critical points to satisfy (A1): by the second bullet point of Proposition~\ref{prop:ems} this is always possible.  Now critical points of the same index can be arranged into any chosen order that satisfies (A2), by the first bullet point of Proposition~\ref{prop:ems} and further applications of Theorem~\ref{thm:ert}.
\end{proof}

The conclusions of Theorem~\ref{thm:egrt} do not hold in codimension~$1$, since critical points of the same index cannot in general be rearranged.
Instead we have the following weaker result.

\begin{theorem}\label{prop:egrt1}
Suppose that $Z,\O$ and $f$ are as in Theorem~\ref{thm:egrt}, but $\codim(\O\subset Z\times[0,1])=1$.
Then
there exists a function $G\colon Z \times [0,1] \to [0,1]$ without critical points, which restricts to a Morse function $g\colon\O\to[0,1]$ having the same critical points as $f$, but the critical values satisfy:
\begin{itemize}
\item[(A1')] if $z,w$ are critical points with indices $k,l$ with $k<l$, then $g(z)<g(w)$.
\end{itemize}
\end{theorem}

\section{Cancellation of critical points}\label{section:cancellation}

In the absolute case, two critical points of indices $k$ and $k+1$ can be cancelled if there is a single trajectory of a Morse--Smale gradient-like
vector field between them. The situation is slightly more complicated in the embedded case.  We present a result which is stated in \cite[Lemma~2.9]{Pe} and \cite[Lemma~5]{Ro}.  For the convenience of the reader we sketch the proof.  The proof also makes crucial use of Vector Field Integration Lemma~\ref{lem:integrate}.

Note that if we have critical points $z,w$ with indices $k,k+1$ respectively, then we can assume that $g(z) < g(w)$, by Theorem~\ref{thm:egrt} in the case that the codimension is~$2$ or more, and Theorem~\ref{prop:egrt1} in the codimension~1 case.

\subsection{The embedded elementary cancellation theorem}\label{section:elem-cancel-thm}

\begin{theorem}[Elementary Cancellation Theorem]\label{lem:eec}
Let $(Z,\O)$ be a cobordism. Let $z$ and $w$ be critical points of $f$ of indices $k$ and $k+1$, of the same type (i.e. either both interior,
or both boundary stable, or both boundary unstable).
Suppose that $\xi$ is an embedded gradient-like vector field, which is Morse--Smale (Definition~\ref{def:emsc}), and that there exists
a \emph{single trajectory} $\gamma$ of $\xi$ connecting $z$ with $w$. If $z$, $w$ are interior critical points, we require that $\gamma\subset\O$.  If $z$, $w$ are boundary critical points, then we require that $\gamma\subset Y$. Furthermore, suppose that there are no broken trajectories between $z$ and $w$.

Then, for any neighbourhood $U$ of $\gamma$, there
exists a vector field $\xi'$ on $Z\times[0,1]$, agreeing with $\xi$ away from $U$, non-vanishing on $U$, and a function
$F'\colon Z\times[0,1]\to[0,1]$ such that $f'=F'|_\O$ has the same critical points as $f$ with the exception of $z$ and $w$,
which are regular points of $f'$, and such that $\xi'$ is an embedded gradient-like vector field with respect to $F'$.
\end{theorem}

\begin{remark}
\leavevmode\Nopagebreak
 \begin{enumerate}
  \item In particular note that the assumptions on a single trajectory and the lack of broken trajectories imply that the intersections of the interiors of the membranes $\Int M^u(z) \cap \Int M^s(w)$ is empty.  In codimension~$3$ or more, such disjointness can always be arranged by general position and Table~\ref{tab:one}: $(n+2-k) + (k+1+1) = n+4 \leqslant m+1$ when $m \geqslant n+3$.  This can be viewed as the main reason why ``concordance implies isotopy''
in codimension~ $3$ or more~\cite{Hudson}.  When the codimension is 1 or 2 there are obstructions from membrane intersections.
  \item Unlike Milnor \cite[Theorem 5.4]{Mi2}, we do not assume that $z$ and $w$ are the only critical points in $f^{-1}[f(z),f(w)]$.  The assumption is replaced by the lack of broken trajectories. The statement is equivalent to that of Milnor in the absolute case, or in the embedded case with codimension 2 or more, since the critical points can be rearranged to achieve Milnor's assumption. It can be shown that it is also equivalent in the case of codimension 1, but since rearrangement of critical points of the same index is in general not possible in codimension 1, it is nice to be able to separate rearrangement and cancellation.
\item The new Morse function $F' \colon Z \times [0,1] \to [0,1]$ can be chosen so as to be homotopic through submersions to $F$.
\end{enumerate}
\end{remark}

\begin{proof}[Proof of Theorem~\ref{lem:eec}]
Milnor's approach works with a few modifications to adapt it to the embedded case and the possibility of additional critical points between $z$ and $w$.
The proof of Theorem~\ref{lem:eec} proceeds by way of Lemmas~\ref{lem:noreturn}, \ref{lem:milnor-ass6} and~\ref{lem:leaves},
and will take the remainder of this section.

\begin{lemma}[c.f.\ \expandafter{\cite[Assertion 1, page 50]{Mi2}}]\label{lem:noreturn}
Let $\gamma$ be as in Theorem~\ref{lem:eec}.  For any open set $U_1\subset Z\times[0,1]$, such that $\gamma\subset U_1$, there exists another open set $U_2$ with $\gamma\subset U_2\subset U_1$
such that any trajectory of $\xi$ which starts in $U_2$ and leads out of $U_1$ never goes back to $U_2$.
\end{lemma}
\begin{proof}[Proof of Lemma~\ref{lem:noreturn}]
The proof resembles that of Lemma~\ref{lem:csn} above, hence we do not give all the details.  The idea is that a trajectory which leaves and then returns will imply the existence of another trajectory between $z$ and $w$, or a broken trajectory, both of which are assumed in the hypothesis of Theorem~\ref{lem:eec} not to exist.

Suppose the statement is false. That is, there exists an open set $U_1$ with $\gamma \subset U_1$ such that for all $U_2$ with $\gamma\subset U_2\subset U_1$ there is a trajectory which starts in $U_2$, leaves $U_1$, and then later returns to $U_2$.  Then there exist sequences of points $\{w_r^1\}_{r \geqslant 1}$, $\{w_r^2\}_{r \geqslant 1}$ and $\{w_r^3\}_{r \geqslant 1}$ and a sequence of trajectories $\{\gamma_r\}_{r \geqslant 1}$,
such that $w_r^1$ and $w_r^3$ approach $\gamma$ as $r\to\infty$, $w_r^2\in\p (\cl(U_1))$ and $\gamma_r$ is a trajectory going first through $w_r^1$, then through $w_r^2$ and finally through $w_r^3$.  By choosing successively smaller open neighbourhoods $U_2^r$ we obtain trajectories $\gamma_r$ from our assumption, and then can choose $w^1_r, w^3_r \in U_2^r \cap \gamma_r$.

Since $\p (\cl{U_1})$ is compact, we may assume that $w_r^2$ converges to a point $w_0$ (as in the proof of Lemma~\ref{lem:integrate} we may need to pass to a subsequence). Furthermore, as $w_r^1$ and $w_r^3$ come very near $\gamma$, we can move $w_r^3$ along $\gamma_r$ to ensure that $w_r^3\to w$. Likewise, we may assume that $w_r^1\to z$.

Consider $\gamma_0$, the trajectory through $w_0$. Since, for any neighbourhood of $z$, there exist $r_0$ such that for all $r>r_0$, $\gamma_r$ enters
this neighbourhood, we would like to claim that $\gamma_0$ starts at $z$. Similarly, we would like to claim that $\gamma_0$ ends at $w$. This is true under the assumption that there are no critical points in $F^{-1}(F(z),F(w))$.

In the general case,
the sequence of trajectories $\gamma_1,\dots,\gamma_r,\dots$ either converges to a trajectory $\gamma_0$ (and then $\gamma_0$ is another trajectory
between $z$ or $w$ by the arguments as above) or to a broken trajectory; see e.g.\ \cite[Corollary 6.23]{Ban}. This broken trajectory is a broken trajectory between
$z$ and $w$, by our assumptions on the limits of $w^1_r$ and $w_3^r$.  This contradicts the assumptions of the Elementary Cancellation Theorem~\ref{lem:eec}, namely that there is a single trajectory joining $z$ to $w$, and no broken trajectories.
\end{proof}

We continue with the proof of Theorem~\ref{lem:eec}.  Let $U_z$ and $U_w$ be open neighbourhoods of $z$ and $w$ respectively, such that the Embedded
Morse Lemma holds for them. Define $U_{zw}$ to be the set of points $x\in Z\times[0,1]\sm (U_z\cup U_w)$ such that the trajectory through $x$ hits $U_z$ in the past and $U_w$ in the future. Shrink $U_{z}$ and $U_w$ if needed, in order to guarantee that $U_{zw},U_z,U_w\subset U$. Indeed, if $V_2$ is defined to be the set $U_2$ given
by Lemma~\ref{lem:noreturn}, with $U_1=U$, and we take $U_z,U_w\subset V_2$, then any trajectory from $U_z$ to $U_w$ must be contained in $U$. The vector field $\xi$ flows from $U_z$ to $U_w$.

The next result is
an analogue of \cite[Assertion 6, page 55]{Mi2}. We state it in the case that $z$ and $w$ are interior critical points; see Remark~\ref{rem:boundary} below
for the boundary case.

\begin{lemma}\label{lem:milnor-ass6}
It is possible to change the vector field $\xi$, inside $U_{zw}$, to a vector field $\xi_1$, which is still a gradient-like vector field for $F$, and
flows from $U_z$ to $U_w$ (that is, a trajectory of $\xi_1$ through a point $p$ in $U_{zw}$ hits $U_z$ in the past and $U_w$ in the future)
such that there is a smaller neighbourhood $U_1\subset \ol{U_z\cup U_w\cup U_{zw}}$ containing $\gamma$ and a coordinate system on $U_1$ given by
$x_1,\dots,x_{n+1},y_1,\dots,y_{m-n}$ such that $z=(0,\dots,0)$, $w=(1,0,\dots,0)$, $\O\cap U_1=\{y_1=\dots=y_{m-n}=0\}$ and
$\xi$ has the following form:
\begin{equation}\label{eq:formofxi}
(v(x_1),-x_2,\dots,-x_{k+1},x_{k+2},\dots,x_{n+1},y_1^2+\dots+y_{m-n}^2,0,\dots,0),
\end{equation}
where $v$ is a smooth function positive on $(0,1)$, negative away from $[0,1]$ and such that $v(x)=x$ (respectively: $v(x)=1-x$) in a neighbourhood of $0$ (respectively: in a neighbourhood of $1$). The coordinate systems on $U$ and $U_z,U_w$ agree (up to shifting $x_1$ by $+1$ on $U_w$).
\end{lemma}
\begin{remark}\label{rem:boundary}
If $z$ and $w$ are boundary critical points then the statement of Lemma~\ref{lem:milnor-ass6} should be changed as in \cite[Proposition 5.2]{BNR1}.
Namely, the local coordinates system should be such that $\O=\{x_1\geqslant 0,y_1=\ldots=y_{m-n}=0\}$, $z=(0,\ldots,0)$, $w=(0,1,0,\ldots 0)$ and
$\xi$ should be of the form
\[(\epsilon x_1,v(x_2),-x_3,\ldots,-x_{k+1},-\epsilon x_{k+2},x_{k+3},\ldots,x_{n+1},y_1^2+\ldots+y_{m-n}^2,0,\ldots,0),\]
where $\epsilon=1$
for boundary unstable and $\epsilon=-1$ for boundary stable critical points.
\end{remark}

\begin{proof}[Sketch of proof of Lemma~\ref{lem:milnor-ass6}]
The proof is analogous to the proof of Assertion 6 in \cite{Mi2}. It consists of looking at the map $h$ between $\p U_z\cap\p U_{zw}$ and
$\p U_w\cap\p U_{zw}$ induced by the flow of $\xi$. One writes it in coordinates of $U_z$ and $U_w$. If this is the identity, we can extend the
coordinates from $U_z$ to $U_{zw}$ using the flow of $\xi$, and the coordinate systems on $U_z$, $U_w$ and $U_{zw}$ match together. Otherwise, we change~$h$ by an isotopy so that it is the identity near the origin and use Embedded Isotopy Lemma~\ref{lem:eil}. We omit the details, which are a straightforward but tedious generalisation of Milnor's approach.
\end{proof}

We resume the proof of Theorem~\ref{lem:eec}. Observe that changing the vector field from $\xi$ to $\xi_1$ does not create any broken trajectories between $z$ and $w$, so Lemma~\ref{lem:noreturn} holds for $\xi'$. We now take $U_1$ to be the set given by Lemma~\ref{lem:milnor-ass6}, and apply Lemma~\ref{lem:noreturn} to $U_1$.  Let $U_2$ be the open set which is the output of Lemma~\ref{lem:noreturn}, and choose a still smaller neighbourhood $U_3$ of $\gamma$.
Let $C$ be the supremum of $v(x_1)$ on $U_3$. Let $\psi$ be
a cut-off function, which is $1$ on $U_3$ and $0$ outside of $U_2$. We define $\xi'=\xi_1-(C+1)\psi\p_{x_1}$.
The vector field $\xi'$ now has no zeros on $U_3$, because the first coordinate of $\xi'$ is negative. On $(U_2\sm U_3)\cap\O$, at least one of the $\p_{x_j}$ coordinates of $\xi'$ is non-zero and on $(U_2\sm U_3)\sm\O$,
the coordinate $\p_{y_1}$ is non-zero. So in fact, $\xi'$ has critical points $z$ and $w$
removed. We want to show that $\xi'$ is a gradient-like vector field of some function. We will need the following result.

\begin{lemma}\label{lem:leaves}
Any trajectory $\gamma'$ through a point $u\in U_2$ leaves $U_1$
in the past and in the future.
\end{lemma}
\begin{proof}
This statement is from \cite[Assertion 2, page 51]{Mi2} and its proof is completely analogous.
\end{proof}

We can now finish the proof of Theorem~\ref{lem:eec}.  By Lemma~\ref{lem:leaves} and Lemma~\ref{lem:noreturn} a trajectory passing
 through $U_2$ remains there only for a finite time and does not go back.  In particular, changing $\xi$ to $\xi'$ does not introduce any ``circular'' broken trajectories. Furthermore the fact that any trajectory only stays in $U_2$ for finite time implies that $\xi'$ has property (AG4).

By Lemma~\ref{lem:integrate}, $\xi'$ is a gradient-like vector field of some function $F'$ without critical
points.
\end{proof}

\section{Splitting of critical points}\label{section:splitting}

The aim of this section is to prove the Elementary and Global Handle Splitting Theorems, in Sections~\ref{section:elem-splitting} and \ref{section:global-handle-splitting} respectively.

We begin by showing that a single interior handle can be split into two half-handles; that is, an interior critical point can be pushed to the boundary and exchanged for one boundary stable critical point and one boundary unstable critical point.  Then we investigate when this theorem can be applied to split all of the interior handles simultaneously.  In codimension at least two, this can be achieved; in codimension one we have a partial result.

\subsection{The embedded elementary handle splitting theorem}\label{section:elem-splitting}

The following result holds for any positive codimension.

\begin{theorem}[Elementary Handle Splitting Theorem]\label{thm:embedded}
Let $(Z,\O)$ be an embedded cobordism with $F\colon Z\times[0,1]\to [0,1]$ a projection and $f=F|_\O$ the underlying Morse function.
Suppose that $f=F|_\O$ has a single interior critical
point at $z$ with index $k\in\{1,\dots,n\}$, $f(z)=\frac12$ and suppose that $z$ can be connected to the set $f^{-1}(1/2)\cap Y$ by a smooth path $\gamma$ contained entirely in $f^{-1}(1/2)$.
Then, for any neighbourhood of $\gamma$ in $Z\times[0,1]$, there exists a function $G\colon Z\times[0,1]\to[0,1]$, homotopic through submersions (Definition~\ref{def:regularhomotopy}) to $F$,
agreeing with $F$ away from that neighbourhood, such that $G$ has no critical points and $g:=G|_\O$ has two boundary critical points $z^s$ and $z^u$ of index $k$ and no interior critical points, where $z^s$ is boundary stable, $z^u$
is boundary unstable and there is a single trajectory of $\nabla g$ on $Y$ going from $z^s$ to $z^u$.
\end{theorem}

\begin{proof}
The proof is an extension of argument in the proof of \cite[Theorem 3.1]{BNR1}. First, by \cite[Proposition 3.5]{BNR1}
there exists $\wtr>0$ and a ``half-disc'' $U\subset\O$ with $\gamma \subset U$
with local coordinates $(x,y,\vu)$, for $x\in[0,3+\wtr)$, $|y|<\wtr$ and $||\vu||^2<\wtr^2$, such
that the coordinates of the critical point are $z=(1,0,\dots,0)$ and $f$ has the form
\[y^3-yx^2+y+\frac12+\vu^2.\]
Here we write $\vu=(u_1,\ldots,u_{n-1})$, $||\vu||^2=\sum u_j^2$ and $\vu^2=\sum \epsilon_ju_j^2$, where $\epsilon_j=\pm 1$ depending on the index of
the critical point $z$.

We thicken $U$ in $Z\times[0,1]$ to a ``half-disc'' $W\subset Z\times[0,1]$ of dimension $m+1$ (i.e.\ codimension 0 in $Z\times[0,1]$).  Choose $U$ and $W$ small enough so that they lie inside the neighbourhood of $\gamma$ referred to in the proof of Theorem~\ref{thm:embedded}.
Now $W$ is diffeomorphic to a product $U\times(-\wtr,\wtr)^{m-n}$, which means that there exists a map
$\vw\colon W\to(-\wtr,\wtr)^{m-n}$, $\vw=(w_1,\dots,w_{m-n})$, so that $U=\{w_1=\dots=w_{m-n}=0\}$ and the collection of functions $(x,y,\vu,\vw)$ forms
a local coordinate system on $W$. In this way we identify $W$ with
$[0,3+\wtr)\times(-\wtr,\wtr)^{n+1}$.

\begin{lemma}\label{lem:positivenormal}
There exists a sign $\epsilon \in \{\pm 1\}$, a choice of index $r\in\{1,\dots,m-n\}$, real numbers $\etp,\taup\in(0,\wtr/2)$
and a smooth function $F_\phi\colon Z\times[0,1]\to[0,1]$, such that $F_\phi|_\O=F|_\O$,
$F_\phi$ agrees with $F$ away from $[0,3+\etp]\times[-2\etp,2\etp]^n\times[-\taup,\taup]^{m-n}\subset W$ and for any $v\in[0,3+\etp]\times[-\etp,\etp]^n\times\{0\}$
we have $\epsilon\frac{\p F_\phi}{\p w_r}(v)>0$.
\end{lemma}
\begin{proof}
As $z$ is not a critical point of $F$ (only of $F|_\O$), there exists an integer $r\in\{1,\dots,m-n\}$ such that
$\frac{\p F}{\p w_r}(z)\neq 0$. We choose $\epsilon$ so that $\epsilon\frac{\p F}{\p w_r}(z)>0$.
By continuity of $\frac{\p F}{\p w_r}$ there exists $\etp>0$, $\etp<\wtr/2$,
such that $\epsilon\frac{\p F}{\p w_r}(v)>0$ whenever $v\in[1-\etp,1+2\etp]\times[-\etp,\etp]^{m}$.
Define:
\begin{align*}
A&=([0,1-\etp]\cup[1+2\etp,3+\etp])\times[-\etp,\etp]^n\subset U\subset\O;\\
A'&=([0,1-\etp/2]\cup[1+\etp,3+\etp])\times[-2\etp,2\etp]^n\subset U\subset\O.\\
\end{align*}

Note that $A \subset A'$.  We choose $\taup$ such that $0 < \taup <\etp$ and define
\begin{align*}
p_1&=2\sup\left\{-\epsilon\frac{\p F}{\p w_r}(v)\colon v\in A\right\}\\
p_2&=\frac12\inf\left\{\left|\frac{\p F}{\p y}(v)\right|\colon v\in A'\times[-\taup,\taup]^{m-n}\subset W\right\}.
\end{align*}
If $p_1<0$ no changes are needed and the proof is finished. If $p_1=0$, we redefine $p_1$ to be a very small positive number.
As for $v\in A'$ we have $\frac{\p F}{\p y}(v)\neq 0$ by direct computation, for $\taup$ small enough we have $p_2>0$. We assume
that
\[\taup<\frac{p_2\etp}{2 p_1}.\]
 If this is not true at first then choose a smaller $\taup$; this lowers the left hand side and cannot lower the right hand side, since only $p_2$ depends on $\taup$ and lower $\taup$ cannot lower $p_2$.
Now choose
a cut-off function $\phi_1\colon Z\times[0,1]\to[0,1]$ with support contained in $A'\times[-\taup,\taup]^{m-n}$
such that $\left|\frac{\p\phi_1}{\p y}\right|<\frac{2}{\etp}$
and $\phi_1|_A=1$. We define now
\begin{equation}\label{eq:defFphi}
F_\phi=F+\epsilon p_1 w_r\phi_1(x,y,\vu,\vw).
\end{equation}
This function agrees with $F$ on $Z\times[0,1]\sm (A'\times[-\taup,\taup]^{n-m})$ and on $A'$ (because on $U$ we have $w_r=0$).
Furthermore, for $v\in A$,
\[\epsilon\frac{\p F_\phi}{\p w_r}(v)=\epsilon\frac{\p f}{\p w_r}(v)+p_1>0,\]
so $\epsilon\frac{\p F_\phi}{\p w_r}>0$ everywhere on $[0,3+\etp]\times[-\etp,\etp]^n\subset U$. To show that $F_\phi$ has no critical
points in $A' \times [-\taup,\taup]^{m-n}$ we compute
\[\left|\frac{\p F_\phi}{\p y}\right|\geqslant
\left|\frac{\p F}{\p y}\right|-\left|p_1w_r\frac{\p\phi_1}{\p y}\right|> 2p_2-p_1\frac{p_2\etp}{2p_1}\frac{2}{\etp} = p_2 >0.\]
\end{proof}

Given Lemma~\ref{lem:positivenormal} we
write $F$ instead of $F_\phi$.  Define
\[W_0(t):=[0,3+\wtr]\times[-\wtr,\wtr]^n\times[-t,t]^{n-m}.\]
Furthermore define
\[p_3(t) := \inf\left\{\epsilon\frac{\p f}{\p w_r}(v)\colon v\in W_0(t)\right\}.\]
For $t$ small enough, by continuity of $\frac{\p f}{\p w_r}$ we have that $p_3(t)>0$.  Choose $\taus>0$ to be such a $t$ and satisfying $\taus < \taup$, where $\taup$ is from Lemma~\ref{lem:positivenormal}.
Now fix $p_3:= p_3(\taus)>0$, and choose $\eta < \etp$, where $\etp$ is also from Lemma~\ref{lem:positivenormal}, satisfying
\[\eta<\frac{\taus p_3}{4}.\]
Let $\delta>0$ be very small. Let us choose a function $b\colon U\to[0,1]$ as in \cite[Equation~(3.7)]{BNR1}.
Let $\phi_2\colon[-\taus,\taus]\to[0,1]$ be a cut-off function, with $\phi_2\equiv 1$
near $0$ and $|\phi_2'(s)|<\frac{2}{\taus}$ for all $s \in [-\taus,\taus]$ (we use the notation $\phi_2$ to avoid confusion with $\phi$ from \cite[Equation~(3.7)]{BNR1}).
Define
\[W_1:=[0,3+\eta]\times[-\eta,\eta]^n\times[-\taus,\taus]^{n-m}.\]
Note that $W_1 \subset W_0(\taus)$.
By analogy with \cite[Equation~(3.8)]{BNR1}, we define
\[G(v)=\begin{cases}
 F(v)& \mbox{if} \ v\not\in W_1\\
F(v)-\phi_2\left(||\vw||\right)(1+\delta)b(x,y,\vu)y& \mbox{if} \  v=(x,y,\vu,\vw)\in W_1,
\end{cases}
\]
where $||\vw||$ is the Euclidean norm of the vector $\vw$.

Since $\phi_2(0)=1$, on $g:=G|_\O$ agrees with the function $g$ from \cite[Theorem 3.1]{BNR1}, so the properties
of the critical points of $g$ from the statement of the present theorem are satisfied. The last condition that we need
to ensure holds is that $G$ has no critical points on $W_0$. But now $|(1+\delta)b|<2$, so
\[\left|\frac{\p}{\p w_r}\left(\phi_2\left(||\vw||\right)
(1+\delta)b(x,y,\vu)\right)\right|<\frac{4}{\taus}.\]
Now $|y|\leqslant\eta<\frac{p_3\taus}{4}$, hence $\epsilon\frac{\p g}{\p w_r}>0$.

Finally we note that the passage from $F$ to $F_\phi$ and then from $F_\phi$ to $G$ can be obtained by a nondegenerate homotopy through submersions; one simply writes $F_t=F+t\epsilon p_1w_r\phi_1(x,y,\vu,\vw)$ for $t\in[0,1]$ and similarly in the definition of $G$. Thus the functions $F$
and $G$ are homotopic through submersions as claimed.
\end{proof}

\subsection{Moving many critical points at once}

Now we pass to the problem of moving all the critical points of a cobordism to the boundary at once. Checking the condition that each critical point
of $f$ can be joined to the boundary by a curve lying entirely in one level set of $f$ is rather complicated.  We shall show that
this is possible.  We would mostly like to repeat the procedure from \cite[Section 4.5]{BNR1}. However in the embedded case we cannot, in general,
cancel pairs of $0$ and $1$ handles, in such a way that we preserve the isotopy class of $\O$, if the codimension is 1 or 2.
Therefore the notion of a technically good function
(see \cite[Definition 4.8]{BNR1}) is not suitable for our present purpose.  We need to modify the reasoning and replace the notion of a technically
good function with a more convenient notion.

\begin{definition}
A function $f\colon\O\to[0,1]$ is called \emph{technically still acceptable} if there exist non-critical values $a,b,c,d$ of $f$ with $0<a<c<d<b<1$ such that
the critical points of $f$ are distributed in the following way. For $n>1$:
\begin{itemize}
\item[(TSA1)] The inverse image $f^{-1}[0,a]$ contains all critical points of index $0$ and all boundary stable critical points of index $1$. It does not contain
any other critical points.
\item[(TSA2)] The inverse image $f^{-1}[a,c]$ contains all interior critical points of index $1$ and no other critical points.
\item[(TSA3)] The inverse image $f^{-1}[d,b]$ contains all interior critical points of index $n$ and no other critical points.
\item[(TSA4)] The inverse image $f^{-1}[b,1]$ contains all critical points of index $n+1$ and all boundary unstable critical points of index $n$.  It does not contain any other critical points.
\item[(TSA5)] The interior critical points of index $1$ lie on a single level set $f^{-1}(\rho)$ and the interior critical points of index
$n$ lie on a single level set $f^{-1}(\rho')$.
\end{itemize}

If $\dim\O=2$, so $n=1$, then for compatibility we modify the conditions (TSA1)--(TSA5). Namely we define $d=a$, $b=c$, and we have $0< a < c < 1$, so that $[c,d]$ is undefined. We assume that the interior critical points of index $1$ lie in $f^{-1}(\rho)$ for some $\rho\in[a,c]$. Furthermore, critical points of index $0$ are in $f^{-1}[0,a]$, as well as boundary stable critical points of index $1$; critical points of index $2$ and boundary unstable critical points of index $1$ are in $f^{-1}[c,1]$.
\end{definition}

We point out that we do not assume that there are no pairs
of critical points that could be cancelled (i.e.\ which are joined by a single trajectory, and no broken trajectories, of the gradient-like
vector field, are of the same type and have indices $k$, $k+1$ for some $k$), and also we do not assume that the critical points whose indices are between~$2$ and~$n-1$ are ordered.
By Theorem~\ref{thm:egrt} we can always assume that the function $f$ is technically still acceptable in codimension two or more.
We have the following analogue of \cite[Proposition 4.11]{BNR1}.

\begin{proposition}\label{prop:4.11}
Suppose that neither $\O$, $\S_0,$ nor $\S_1$ have closed connected components.
If $\codim\O\subset Z\times[0,1]\geqslant 2$ and $f$ is technically still acceptable, then we can rearrange $F$ so that each interior critical
point $z$ of $f=F|_\O$ of index $1,\dots,n$ can be joined to the boundary by a curve lying entirely in $f^{-1}(f(z))$.
\end{proposition}
\begin{proof}[Beginning of the proof of Proposition~\ref{prop:4.11}]
We shall follow the proof of \cite[Proposition 4.11]{BNR1}. In particular we shall use several lemmas from that paper.  The majority of the proof of Proposition~\ref{prop:4.11} will be contained in Lemmata~\ref{lem:4.13} through~\ref{lem:new}.

Observe that in the present
case the interval $[c,d]$ does not contain the critical value of any interior critical point with index $1$ or $n$.  First we shall prove that for any $y\in[c,d]$, the inverse image $f^{-1}(y)$ has no closed connected components, then we shall
work with $f^{-1}[a,c]$ and $f^{-1}[b,d]$.

Let us recall the following lemmas from \cite{BNR1}. All the proofs are given in that article; we indicate how they can be modified for our embedded case as necessary.

\begin{lemma}[see \expandafter{\cite[Lemma 4.13]{BNR1}}]\label{lem:4.13}
Let $x,y\in[0,1]$ with $x<y$. If $\O'$ is a connected component of $f^{-1}[x,y]$ then either $\O'\cap Y=\emptyset$, or for any $u\in[x,y]\cap[c,d]$
we have $f^{-1}(u)\cap\O'\cap Y\neq\emptyset$.
\end{lemma}
\begin{lemma}[see \expandafter{\cite[Lemma 4.14]{BNR1}}]\label{lem:4.14}
For any $x\in[c,1]$ the set $f^{-1}[0,x]$ cannot have a connected component disjoint from $Y$.
\end{lemma}
\begin{proof}[Proof of Lemma~\ref{lem:4.14}]
The proof in \cite{BNR1} relies on having cancelled all possible pairs of critical points of indices $0$ and $1$. We
present a modification of that proof which does not assume this.

Suppose for a contradiction that $\O'$ is a connected component of $f^{-1}[0,x]$ such that $\O'\cap Y=\emptyset$. Let~$\O_1$
be the connected component of $\O$ which contains $\O'$. If $\O_1\cap Y=\emptyset$, then either $\p\O_1=\emptyset$ or $\p\O\subset\S_0\cup\S_1$.
In the first case $\O_1$ is a closed connected component of $\O$.  In the second either $\S_0$ or $\S_1$ has a closed connected component.  This contradicts the hypothesis of Proposition~\ref{prop:4.11}.  The contradiction
implies that $\O_1\cap Y\neq\emptyset$. By Lemma~\ref{lem:4.13}, $f^{-1}(x)\cap\O_1\cap Y\neq\emptyset$.
In particular $\O'':=(f^{-1}[0,x]\cap\O_1)\sm\O'$ is not empty.

Now, $\O''$ and $\O'$ are both subsets of a connected space $\O_1$. Thus, there must be a critical point $z\in\O_1$, of index $1$, which
joins $\O'$ to $\O''$. We have that $f(z)>x$. As the connected component of $f^{-1}[0,f(z))$ containing $\O'$ has empty intersection with $Y$
(by Lemma~\ref{lem:4.13}), $z$ must be an interior critical point.

Up until now we were following the proof of \cite[Lemma 4.14]{BNR1}. Now we use a different argument. Namely, as $x>c$, $f(z)>c$ as well.
But the property (TSA2) implies that there cannot be any interior critical points of index $1$ in $f^{-1}[c,1]$. This contradiction finishes the proof of the lemma.
\end{proof}

\begin{lemma}[see \expandafter{\cite[Lemma 4.16]{BNR1}}]\label{lem:4.16}
Let $y\in[c,d]$ be chosen so that there are no interior or boundary unstable critical points of index $n$ with critical values in $[c,y)$.
Then $f^{-1}(y)$ has no closed connected components.
\end{lemma}

Observe that in the present situation, (TSA3) implies that the assumption to Lemma~\ref{lem:4.16} are automatically satisfied for $y=d$. Hence
we get the following result.

\begin{lemma}\label{lem:new2}
For any $y\in[c,d]$ the set $f^{-1}(y)$ has no closed connected components.
\end{lemma}

We are going to deal with the set $f^{-1}(y)$ for $y\in[a,c]$. The case $y\in[d,b]$ is symmetric. The following lemma holds for $n>1$ and $n=1$, although
the proof in the two cases is different.

\begin{lemma}\label{lem:new}
The interior critical points of $f$ of index $1$ can be rearranged, without changing $f$ away from $f^{-1}[a,c]$, so that
each critical point $z$ of index $1$ can be connected to the boundary of $Y$ with a curve lying entirely in $f^{-1}(z)$.
\end{lemma}
\begin{proof}
The proof follows \cite[proof of Proposition 4.11, case $n=1$]{BNR1} with a small modification. Suppose $n>1$.

Observe that $f^{-1}(c)$ has no closed connected components by Lemma~\ref{lem:new2}. The interior critical points of index $1$ all lie originally
on the level set $f^{-1}(\rho)$ by (TSA5), for some $\rho\in[a,c]$. It follows that $f^{-1}(\rho)$ has no closed connected components, in fact
$f^{-1}(\rho)$ arises from $f^{-1}(c)$ by contracting each intersection $W^u(z)\cap f^{-1}(c)$ to a point, where $z$ ranges through all interior critical points of index $1$.

Now we proceed by induction, as in \cite{BNR1}.  Assume that the interior critical points of index $1$
are $z_1,\dots,z_k$. Let us choose $a_1,\dots,a_k$ such that $a<a_1<\dots<a_k<\rho$ (this is different from~\cite{BNR1}).
Choose a critical point $z\in f^{-1}(\rho)$ which can be connected to $Y$ by a curve $\gamma$ staying in $f^{-1}(\rho)$
and not intersecting other critical points. Assume that this is $z_1$. Then we rearrange the critical points, so that $f(z_1)=a_1$ and the
position of other critical points is not changed.  We claim that $f^{-1}(\rho)$ still has no closed connected components. This is so, because
if $n>1$ a 1-handle does not increase the number of connected components.  Hence we find again a critical point in $f^{-1}(\rho)$, which we call $z_2$, that can be connected to $Y$ by a curve lying entirely in $f^{-1}(\rho)$ omitting $z_3,\dots,z_k$. We move $z_2$ to the level $a_2$.

After finite number of moves we complete the proof.

\smallskip
The $n=1$ case is slightly more difficult, because the number of connected components of the level set can increase or decrease depending on
whether the ascending sphere $S^0$ belongs to a single connected component of $f^{-1}(a)$, or to two components. We will use the
following trick.  As before let the critical points be $z_1,\dots,z_k$. Let us choose numbers $a_1,\dots,a_k,c_1,\dots,c_k$ such that
$a<a_1<\dots<a_k<\rho<c_k<\dots<c_1<c$. Recall that $f^{-1}(\rho)$ has no closed connected components. Let us again choose a critical point (we relabel
the critical points so that this is $z_1$), which
can be connected to $Y$ in the level set of $f^{-1}(\rho)$ by a curve omitting all other critical points. We have two cases: either $z_1$ is attached to
two separate connected components of $f^{-1}(a)$, then we move it to the level $a_1$; or $z_1$ is attached to a single connected component,
and then we move $z_1$ to the level set of $c_1$. In both cases, it still holds that $f^{-1}(\rho)$ has no closed connected component. We can thus find
a critical point in $f^{-1}(\rho)$ (we will call it $z_2$) that can be connected to the $Y$ by a path in $f^{-1}(\rho)$ omitting other critical points.
We move it to the level set $a_2$ or $c_2$ as above. The procedure is then repeated inductively with the remaining critical points.
\end{proof}

\emph{Conclusion of the proof of Proposition~\ref{prop:4.11}.} For $n>1$, by Lemma~\ref{lem:new2} $f^{-1}(y)$ has no closed connected components
for any $y\in[c,d]$.  Hence all interior critical points of index $2,\dots,n-1$ can be connected to the boundary by a curve lying
in the corresponding level set. The case of interior critical points of index $1$ is dealt with in Lemma~\ref{lem:new}. The same lemma also deals
with the case of critical points of index $n-1$ by applying it to $1-f$; the proof for $n>1$ is finished.

For $n=1$ we use only Lemma~\ref{lem:new}.
\end{proof}

If the codimension of the embedding is one we have the following partial result.

\begin{proposition}\label{prop:codim1-4.11}
Suppose that neither $\O$, $\S_0$, nor $\S_1$ have closed connected components and $\codim\O\subset Z\times[0,1]=1$. Then we can change $F$ by a rearrangement so that each critical point $z$  of $f=F|_\O$ of index $2,\ldots,n-1$ can be joined to the boundary by a curve lying entirely in the
level set $f^{-1})(f(z))$.
\end{proposition}
\begin{proof}
By Theorem~\ref{prop:egrt1} (Global Rearrangement Theorem in codimension one) we rearrange~$F$ so that if~$z$ and~$w$ are critical points of $f=F|_\O$
and the index of $z$ is smaller than the index of $w$, then $F(z)<F(w)$. In particular we can choose $c<d$ such that $F^{-1}[c,d]$ contains
all critical points of $f$ of indices between~$2$ and~$n-1$ inclusive, and only those critical points.

Now Lemmata~\ref{lem:4.13}, \ref{lem:4.14} and \ref{lem:4.16} hold in this case, because the codimension assumption is never used in the proofs.
\end{proof}

\subsection{The embedded global handle splitting theorem}\label{section:global-handle-splitting}

Now we are going to prove one of the main results of the present paper, which is the topological counterpart of \cite[Main Theorem 1]{BNR2}.

\begin{theorem}[Global Handle Splitting Theorem]\label{thm:main1}
Let $(Z,\O)$ be an embedded cobordism such that $\O\subset Z\times[0,1]$ has codimension 2 or more. Suppose that $\O$, $\S_0$ and $\S_1$ have
no closed connected components. Then there exists a map $F\colon Z\times[0,1]\to [0,1]$, which is homotopic through submersions to the projection onto the second factor, such that $\O$ can be expressed as a union:
\[\O=\O_{-1/2}\cup \O_0\cup \O_{1/2}\cup \O_1\cup \O_{3/2}\cup\dots\cup \O_{n+1/2}\cup \O_{n+1},\]
where $\O_i=\O\cap F^{-1}([(2i+1)/(2n+4),(2i+2)/(2n+4)])$ and
\begin{itemize}
\item[$\bullet$] $\O_{-1/2}$ is a cobordism given by a sequence of index $0$ handle attachments;
\item[$\bullet$] if $i\in\{0,\dots,n\}$, then $\O_i$ is a right product cobordism given by a sequence of elementary index $i$ right product cobordisms;
\item[$\bullet$] if $i+1/2\in \{1,\dots,n+1\}$, then $\O_i$ is a left product cobordism given by a sequence of elementary index $i+1/2$ left product cobordisms;
\item[$\bullet$] $\O_{n+1}$ is a cobordism given by a sequence of index $n+1$ handle attachments.
\end{itemize}
\end{theorem}
\begin{proof}
The proof follows the line of \cite[Theorem 4.18]{BNR1} with the exception that we cannot in general assume that the original underlying Morse
function has only boundary stable critical points.

By the Global Rearrangement Theorem (Theorem~\ref{thm:egrt}) we can rearrange the critical points of $f$ so
that the boundary stable critical points of index $k$ have critical value
$\frac{3k+1}{3n+6}$, interior critical points of index $k$ are on the level $\frac{3k+2}{3n+6}$ and boundary unstable critical points of
index $k$ are on the level $\frac{3k+3}{3n+6}$. We point out that there are no boundary unstable critical points of index $n+1$,
nor boundary stable critical points of index $0$ (see Remark~\ref{remark:dbend-index-shift}).

After such rearrangements, the function $f$ is technically still acceptable. By Proposition~\ref{prop:4.11} we can join each interior critical point
of index $1,\dots,n$ to the boundary by a curve lying in a level set of $f$. Then, by Theorem~\ref{thm:embedded}
we can move these interior critical points to the boundary and split into boundary stable and unstable critical points. After this, the critical points
are organized so that first come (this means that the value of $f$ at the corresponding points is the smallest)
interior index $0$ critical points, then boundary unstable index $0$, then boundary stable index $1$, boundary unstable index $1$ and so on.
Finally we have boundary unstable critical points of index $n$, boundary stable critical points of index $n+1$ and interior critical points of
index $n+1$.

We can now choose noncritical values of $f$, $0=f_0<f_1<f_2<\dots<f_{2n+4}=1$ in such a way that $f^{-1}[0,f_1]$ contains only
the interior critical points of index $0$, $f^{-1}[f_{2i+1},f_{2i+2}]$ contains the boundary unstable critical points of index $i$ for $i=0,\dots, n$, $f^{-1}[f_{2i},f_{2i+1}]$ contains the boundary stable critical points of index $i$ for $i=1,\dots,n+1$ and $f^{-1}[f_{2n+3},f_{2n+4}]$ contains only the interior critical points of index $n+1$. We define $\O_i :=f^{-1}[f_{2i+1,2i+2}]$ (we can rescale the function $F$ so that $f_{j}=\frac{j}{2n+4}$).
\end{proof}

In the codimension one case, the situation is much more complicated. First, the rearrangement theorem does not work in general.  Since we are not
allowed to switch the position of critical points of the same index, the method of proof used for Proposition~\ref{prop:4.11} for the interval $[a,c]$
does not work. For the moment we can prove the following result.

\begin{theorem}\label{prop:codim-one-splitting}
Let $(Z,\O)$ be a codimension one embedded cobordism, where $\O$, $\S_0$ and $\S_1$ have no closed connected components.
Then the cobordism $(Z,\O)$
can be split into cobordisms $\O_0,\O_1,\dots,\O_M$ for some $M>0$, where $\O_0$ contains only critical points of index $0$ and $1$ (of all possible types),
$\O_M$ has only critical points of index $n$ and $n+1$ and for $i=1,\dots,M-1$, the cobordism $\O_i$ is either a left product or a right product cobordism, and the index of critical points in $\O_i$ is less than or equal to the index of the critical points in $\O_j$ whenever $1<i <j<M$.
\end{theorem}

We remark that the statement is vacuous if $\dim\O=2$.

\begin{proof}
By Theorem~\ref{prop:egrt1} we can rearrange critical points so that the critical points of index~$0$ come first, followed by the critical points of index $1$, then those of index $2$, and so on. Recall that we cannot, in general, rearrange critical points so that boundary unstable critical points come after other critical points of the same index.
Notwithstanding, there exist $c,d$ with $0<c<d<1$ such that $f^{-1}[c,d]$ contains critical points of indices between $2$
and $n-1$ and no other. We can already define $\O_0=f^{-1}[0,c]$ and $\O_M=f^{-1}[d,1]$. Now by Proposition~\ref{prop:codim1-4.11} we can connect each interior critical point of index $k$, $k=\{2,\ldots,n-1\}$ to the boundary $Y$ by a curve lying entirely in a level set of $f$. Hence we can apply Theorem~\ref{thm:embedded} to split the critical point. Now we define cobordisms $\O_1,\dots,\O_{M-1}$ by the condition that each contains exactly one critical point of $f$.
\end{proof}

The codimension~1 case, although the most difficult, is also very important in applications; see \cite{BNR2}. We conclude the
section with a rough description of how Theorem~\ref{prop:codim-one-splitting} can replace \cite[Main Theorem 1]{BNR2} in the proof of
\cite[Main Theorem 2]{BNR2}.  The former states that relative algebraic cobordisms are algebraically split.  The latter states that isotopic knots have $S$-equivalent Seifert forms and that $H$-cobordant knots have $H$-equivalent Seifert forms (we refer to \cite{BNR2} for the definitions).

Suppose $Z=S^{2k+1}$ is a sphere and $N_0$, $N_1$ are $(2k-1)$-dimensional
closed oriented submanifolds of $Z$ with Seifert surfaces $\S_0$ and $\S_1$ respectively.  The Seifert forms are defined on the
torsion--free parts of $H_k(\S_0,\Z)$ and $H_k(\S_1,\Z)$. Suppose $N_0$ and $N_1$ are cobordant as submanifolds of $Z \times [0,1]$, for example if the knots are isotopic or concordant. We want to compare the Seifert forms related to $\S_0$ and $\S_1$. To this end, we
find a $(2k+1)$-dimensional manifold
$\O\subset Z\times[0,1]$, such that $\p\O=\S_0\cup Y\cup\S_1$, where $Y$ is the cobordism between $N_0$ and $N_1$. The changes between the Seifert
forms related to $\S_0$ and $\S_1$ can be studied by splitting the cobordism $\O$ into simple pieces, and looking at change corresponding to each piece. If the piece consists of a handle attachment of index $s\neq k,k+1$, then the Seifert form is unchanged; see \cite{Levine:1970-1}. The action occurs in the middle dimensions~$k$ and~$k+1$.
The Seifert form can change, but it is much easier to control the change if the cobordism corresponds to a half-handle attachment
(that is, it corresponds to crossing a boundary critical point) and not a handle attachment (corresponding to crossing an
interior critical point). So to prove \cite[Main Theorem 2]{BNR2}, one wants to move all interior
critical points of index $k,k+1$ to the boundary. By Theorem~\ref{prop:codim-one-splitting}, this is possible if $k\ge 2$, so in any dimension past the classical dimension $k=1$, $Z=S^3$. Of course, in the classical case, the theorem that isotopy of knots implies $S$-equivalence of Seifert matrices, and that Seifert matrices of concordant knots are algebraically concordant, was known long before~\cite{BNR2}.

\bibliographystyle{amsalpha}
\def\MR#1{}
\bibliography{research}

\end{document}